\documentclass[12pt,british]{article}

\usepackage{babel}
\usepackage{mathrsfs,diagbox} 
\usepackage{amsmath,amsfonts,amssymb,amsthm}
\usepackage{ucs} % Unicode support
\usepackage[utf8x]{inputenc} % UCS' UTF-8 driver is better than the LaTeX kernel's
\usepackage[T1]{fontenc} % The default font encoding only contains Latin characters
\usepackage{enumerate}
\usepackage{textcomp,url}
\usepackage{eucal}
\usepackage{comment}
\usepackage[noBBpl]{mathpazo}
\usepackage[matrix,arrow]{xy}
\usepackage{epsfig}
\usepackage[pdftex,                %
implicit=true,
bookmarks=true,
breaklinks=true,
bookmarksnumbered = true,%     % Signets numérotés
colorlinks= true,%     % Liens en couleur
urlcolor= green,
anchorcolor = yellow,
citecolor=blue,%  % Couleur des liens externes
pdfborder= {0 0 0}%   % Style de bordure : ici, pas de bordure
]{hyperref}%                   % Utilization de HyperTeX

\usepackage{ytableau} %to write Young tableaux
\usepackage{multicol}%Oscar: I introduce this package to use in group presentations

\newcommand{\out}[1]{\ensuremath{\operatorname{\text{Out}}\left({#1}\right)}}
\newcommand{\inn}[1]{\ensuremath{\operatorname{\text{Inn}}\left({#1}\right)}}
\makeatletter

\renewcommand{\p@enumii}{}
\def\@enum@{\list{\csname label\@enumctr\endcsname}%
{\usecounter{\@enumctr}\def\makelabel##1{
\normalfont\ignorespaces\emph{{##1}~}}
\setlength{\labelsep}{3pt}
\setlength{\parsep}{0pt}
\setlength{\itemsep}{0pt}
\setlength{\leftmargin}{0pt}
\setlength{\labelwidth}{0pt}
\setlength{\listparindent}{\parindent}
\setlength{\itemsep}{0pt}
\setlength{\itemindent}{0pt}
\topsep=3pt plus 1pt minus 1 pt}}
\makeatother

\renewcommand{\epsilon}{\ensuremath{\varepsilon}}
\renewcommand{\phi}{\ensuremath{\varphi}}

\renewcommand{\to}{\ensuremath{\longrightarrow}}

\newcommand{\R}{\ensuremath{\mathbb R}}

\newcommand{\Z}{\ensuremath{\mathbb Z}}
\newcommand{\Q}{\ensuremath{\mathbb Q}}

\renewcommand{\ker}[1]{\ensuremath{\operatorname{\text{Ker}}\left({#1}\right)}}

\newcommand{\aut}[1]{\ensuremath{\operatorname{\text{Aut}}\left({#1}\right)}}

\newcommand{\aff}[1]{\ensuremath{\operatorname{\text{Aff}}\left({#1}\right)}}

\makeatletter
\def\@map#1#2[#3]{\mbox{$#1 \colon\thinspace #2 \to #3$}}
\def\map#1#2{\@ifnextchar [{\@map{#1}{#2}}{\@map{#1}{#2}[#2]}}
\makeatother

\newcommand{\ang}[1]{\ensuremath{\left\langle #1\right\rangle}}

\newcommand{\setang}[2]{\ensuremath{\ang{#1 \,\mid\, #2}}}

\newtheoremstyle{theoremm}{}{}{\itshape}{}{\scshape}{.}{ }{}
\theoremstyle{theoremm}

\newtheorem{thm}{Theorem}
\newtheorem{lem}[thm]{Lemma}
\newtheorem{prop}[thm]{Proposition}
\newtheorem{cor}[thm]{Corollary}

\newtheoremstyle{remark}{}{}{}{}{\scshape}{.}{ }{}
\theoremstyle{remark}

\newtheorem{defn}[thm]{Definition}
\newtheorem{rem}[thm]{Remark}

\newtheoremstyle{comment}{}{}{\bfseries}{}{\bfseries}{:}{ }{}
\theoremstyle{comment}

\newcommand{\redef}[1]{Definition~\protect\ref{def:#1}}
\newcommand{\rethm}[1]{Theorem~\protect\ref{thm:#1}}
\newcommand{\relem}[1]{Lemma~\protect\ref{lem:#1}}
\newcommand{\reprop}[1]{Proposition~\protect\ref{prop:#1}}

\newcommand{\resec}[1]{Section~\protect\ref{sec:#1}}
\newcommand{\resubsec}[1]{Subsection~\protect\ref{subsec:#1}}
\newcommand{\rerem}[1]{Remark~\protect\ref{rem:#1}}

\newcommand{\req}[1]{equation~(\protect\ref{eq:#1})}
\newcommand{\reqref}[1]{(\protect\ref{eq:#1})}
\allowdisplaybreaks

\usepackage{listings}
\lstdefinelanguage{GAP}{%
  morekeywords={%
    Assert,Info,IsBound,QUIT,%
    TryNextMethod,Unbind,and,break,%
    continue,do,elif,%
    else,end,false,fi,for,%
    function,if,in,local,%
    mod,not,od,or,%
    quit,rec,repeat,return,%
    then,true,until,while%
  },%
  sensitive,%
  morecomment=[l]\#,%
  morestring=[b]",%
  morestring=[b]',%
}[keywords,comments,strings]

\usepackage[T1]{fontenc}
\usepackage[variablett]{lmodern}
\usepackage{xcolor}
\begin{document}

\title{Characteristic subgroups and the R$_\infty$-property for virtual braid groups}

\author{KAREL~DEKIMPE\\
KU Leuven Campus Kulak Kortrijk,\\
Etienne Sabbelaan 53, 8500 Kortrijk, Belgium.\\
e-mail:~\texttt{karel.dekimpe@kuleuven.be}\vspace*{4mm}\\
DACIBERG~LIMA~GON\c{C}ALVES\\
Departamento de Matem\'atica - IME-USP,\\
Rua~do~Mat\~ao~1010~CEP:~05508-090  - S\~ao Paulo - SP - Brazil.\\
e-mail:~\texttt{dlgoncal@ime.usp.br}\vspace*{4mm}\\
OSCAR~OCAMPO~\\
Universidade Federal da Bahia,\\
Departamento de Matem\'atica - IME,\\
Av. Adhemar de Barros~S/N~CEP:~40170-110 - Salvador - BA - Brazil.\\
e-mail:~\texttt{oscaro@ufba.br}
}

\date{\today}

\maketitle

\begin{abstract}
Let $n\geq 2$. Let $VB_n$ (resp.\ $VP_n$) denote the virtual braid group (resp.\ virtual pure braid group), let $WB_n$ (resp.\ $WP_n$) denote the welded braid group (resp.\ welded pure braid group) and let $UVB_n$ (resp.\ $UVP_n$) denote the unrestricted virtual braid group (resp.\ unrestricted virtual pure braid group). 
In the first part of this paper we prove that, for $n\geq 4$, the group $VP_n$ and for $n\geq 3$ the groups $WP_n$ and $UVP_n$ are characteristic subgroups of $VB_n$, $WB_n$ and $UVB_n$, respectively. In the second part of the paper we show that, for $n\geq 2$, the virtual braid group $VB_n$, the unrestricted virtual pure braid group $UVP_n$, and the unrestricted virtual braid group $UVB_n$ have the R$_\infty$-property. As a consequence of the technique used for few strings we also prove that, for $n=2,3,4$, the welded braid group $WB_n$ has the R$_\infty$-property and that 
for $n=2$ the  corresponding pure braid groups  have the R$_\infty$-property. On the other hand for $n\geq 3$  it is   
 unknown if  the R$_\infty$-property holds or not for  the virtual pure braid group $VP_n$ and  the welded pure braid group $WP_n$.

 \end{abstract}

	\let\thefootnote\relax

\section{Introduction}

In this paper we are interested in characteristic subgroups and  the  R$_\infty$-property of the  virtual braid groups.

The virtual braid group $VB_n$ is the natural companion to the category of virtual knots, just as the Artin braid group is to usual knots and links. 
We note that a virtual knot diagram is like a classical knot diagram with one extra type of crossing, called a virtual crossing. The virtual braid groups have interpretations in terms of diagrams, see~\cite{Kam2},~\cite{Kau} and~\cite{V}.
The notion of virtual knots and links was introduced by Kauffman together with virtual braids in \cite{Kau}, and since then it has drawn the attention of several researchers.
Virtual braid groups have interesting quotients, among them we are interested here in the welded braid group $WB_n$ and the unrestricted virtual braid group $UVB_n$. 
As in the classical case for Artin braid groups, notable subgroups of $VB_n$, $WB_n$ and $UVB_n$ are the respective pure subgroups $VP_n$, $WP_n$ and $UVP_n$. For $VB_n$ there is also another notable subgroup that is an Artin group and will be denoted by $KB_n$. 
 For the definitions of all these groups see \resubsec{vbn}.

Let $K$ be a subgroup of a group $G$. Recall that $K$ is said to be a characteristic subgroup of $G$ if $\phi(K)=K$ for every automorphism $\phi$ of $G$. The existence of characteristic subgroups of a given group is in many cases useful. E.g.\ in this paper we will use characteristic subgroups  to 
study the R$_{\infty}$-property  (see the definition below) 
for the virtual braid groups as well for some quotients    of it.

Let $G, H$ be two groups. For every $h\in H$ we have the inner automorphism $c_h\colon H\to H$, defined by $c_h(x)=hxh^{-1}$.  We say that two homomorphisms $\psi_1,\, \psi_2\colon G\to H$ are \emph{conjugate}, and we denote it by $\psi_1\sim_{c} \psi_2$, if there exists an element $h\in H$ such that $\psi_2 = c_h\circ \psi_1$, which means that $\psi_2(g) = h \psi_1(g) h^{-1}$, for every $g\in G$. 
We note that $\sim_{c}$ is an equivalence relation. 
Our first result is the following theorem which gives a condition on the kernel of a homomorphism to be a characteristic subgroup.

\begin{thm}\label{thm:mainchar}
Let $G$ and $Q$ be two groups. Let $\Sigma$ be the set of all surjective homomorphisms from $G$ onto $Q$, let ${\cal T}=\Sigma/\!\sim_{c}$ be the set of equivalence classes of $\Sigma$ by $\sim_{c}$ and let $\Lambda$ be a set of representatives of ${\cal T}$.  
Let $\lambda\in \Lambda$ be such that for all $\omega\in \Lambda$ such that $Ker(\omega)$ is isomorphic to $ Ker(\lambda)$ it actually holds that $Ker(\lambda) = Ker(\omega)$.\\
Then, $Ker(\lambda)$ is a characteristic subgroup of $G$.
\end{thm}

Our main result about characteristic subgroups of virtual braid groups and some of its quotients is the following result.

\begin{thm}\label{thm:maincharvbn}
Let $n\geq 2$.
\begin{enumerate}

	\item The virtual pure braid group $VP_n$ is a characteristic subgroup of $VB_n$ if and only if $n\geq 4$ and the group $KB_n$ is a  characteristic subgroups of the virtual braid group $VB_n$ if and only if $n\geq 3$.

	\item The welded pure braid group $WP_n$ is a characteristic subgroup of the welded braid group $WB_n$ if and only if $n\geq 3$.

	\item The unrestricted virtual pure braid group $UVP_n$ is a characteristic subgroup of the unrestricted virtual braid group $UVB_n$ if and only if $n\geq 3$.
\end{enumerate}
\end{thm}

We note that \rethm{maincharvbn} is known for $n\geq 5$, but was not known for the case of few strings, so in this work we complete the knowledge about pure subgroups being characteristic in their respective virtual braid groups.  
\rethm{maincharvbn} (for $n\geq 5$) was proved for $VB_n$ in \cite{BP} using the explicit description of the automorphism group of $VB_n$, see \cite[Corollary~2.7]{BP}; for $WB_n$ (see \cite[Remark~2.17]{M}) and $UVB_n$ (see \cite[Proposition~2.15]{M}) the description of the automorphism group was not used.  
However, we highlight that our approach is different by using the general result in \rethm{mainchar} for any number of strings: instead of using explicitly the automorphism group of each of the groups involved we used, up to conjugation, the set of surjective homomorphisms onto the symmetric group. The description of these sets for small values on $n$ is given in \resec{char}. 
It is worth to notice that \rethm{maincharvbn} will be used in the proof of \rethm{mainresult} below.

\medskip

Consider a group $G$ and an endomorphism $\alpha$ of $G$. We say that two elements $x$ and $y$ of $G$ are twisted conjugate (via $\alpha$) if and only if there exists a $z\in G$ such that $x = z y \alpha(z)^{-1}$. It is easy to see that the relation of being twisted conjugate is an equivalence relation and the number of equivalence classes (also referred to as Reidemeister classes) is called the Reidemeister number $R(\alpha)$ of $\alpha$. This Reidemeister number is either a positive integer or $\infty$.

Reidemeister numbers find their origins in algebraic topology and to be more precise in Nielsen--Reidemeister fixed point theory. 
Here one is interested in the study of the fixed point classes of a selfmap $f$ of a space $X$. 
The number of fixed point classes of $f$ is called the Reidemeister number of $f$ and is denoted by $R(f)$. It is known that 
 $R(f)=R(f_\ast)$, where $f_\ast\colon \pi_1(X) \to \pi_1(X) $ is the induced endomorphism on the fundamental group $\pi_1(X)$ of $X$.

There is currently a growing interest in the study of groups $G$ having the R$_\infty$-property, these are groups for which $R(\alpha)=\infty$ for any automorphism $\alpha \in \aut{G}$. The study of groups with that  property was initiated by Fel'shtyn and Hill \cite{FH}.

Since the beginning of this century many authors have been studying this property  and for several families of groups it is known whether or not they have the R$_\infty$-property. 
 Here are  some families of groups with the R$_\infty$-property:  the non-elementary Gromov hyperbolic groups \cite{F,LL}, most of the Baumslag--Solitar groups \cite{FG1}  and groups quasi--isometric to Baumslag--Solitar groups \cite{TW},\
 generalized Baumslag--Solitar groups \cite{L}, many linear groups \cite{FN,N} and also several families of lamplighter groups \cite{GW,T}.
The study of the R$_\infty$-property for braid groups and braid-type groups has been increasing during the last years. For instance, in \cite{FG} it was shown that the Artin braid groups $B_n$ and the mapping class groups of closed orientable surfaces different from the sphere have property R$_\infty$. 
In \cite{DGO} the case of the pure Artin braid groups $P_n$ was considered, and they also share this property. 
More recently, the R$_\infty$-property was studied for some right angled Artin groups in  \cite{DS} and for some Artin groups in \cite{CS}.

In this paper we study the R$_{\infty}$-property for virtual braid groups and unrestricted virtual braid groups.   
Since these groups are trivial when $n=1$, we shall consider, in general, $n\geq 2$. 
More precisely, the statement below summarises the main results  in this work about this property.

\begin{thm}\label{thm:mainresult}
Let $n\geq 2$. The following groups have the R$_{\infty}$-property:
\begin{enumerate}
\item \label{item:a} The unrestricted virtual pure braid group $UVP_n$,
\item \label{item:b} the unrestricted virtual braid group $UVB_n$, 
\item \label{item:c} the virtual braid group $VB_n$.\\
Further, if $n=2, 3 \ and \   4$ also the  welded  braid group $WB_n$ has the R$_{\infty}$-property,
and for $n=2$ also the virtual ($VP_2$), welded ($WP_2$) and unrestricted ($UVP_2$) pure braid groups  have the R$_{\infty}$-property.
\end{enumerate}
\end{thm}

For $n\geq 5$,  part $(c)$ of the Theorem \ref{thm:mainresult} above was simultaneously  obtained by N.~Nanda (\cite{Na}), using a different approach.

\medskip
 
This paper is organized as follows. In the first subsection of \resec{prelim} we will give the main definitions about the virtual braid groups that will be used in the text and in the second subsection we prove \rethm{mainchar}. 
In \resec{char} we prove \rethm{maincharvbn}. To do that,  we first treat the case $n=2$ and thereafter we describe the set of homomorphisms, up to conjugation, from $VB_n$ (and also from $WB_n$ and $UVB_n$) to $S_n$, for $n=3,4$.  
Then, in \resubsec{thm2} we use this information and \rethm{mainchar} to prove \rethm{maincharvbn}. 
We prove \rethm{mainresult} in \resec{vbn}, its proof is given in several different steps and using different techniques. 
For $n\geq 2$, the group $UVP_n$ is isomorphic to a direct product of free groups and from \rethm{maincharvbn} it is a characteristic subgroup of $UVB_n$. Using this, in Proposition~\ref{prop:uvbn}, we prove item~(\ref{item:a}) and item~(\ref{item:b}) of \rethm{mainresult}.
To prove \rethm{mainresult}~(\ref{item:c}), for $n\geq~5$, we show that the kernel of the natural projection of $VB_n$ onto $UVB_n$ is a characteristic subgroup and then we use item~(\ref{item:b}) of \rethm{mainresult}, this will be done in Theorem \ref{main}.
For $n=3$ or $n=4$, we first show that the 
quotient $VB_n/[VP_n,\, VP_n]$ has the R$_{\infty}$-property, where 
$[VP_n,\, VP_n]$ is the commutator subgroup of the virtual pure braid group $VP_n$, see Theorems~ 
\ref{thm:vb3cryst} and~\ref{thm:crystgroup}, respectively. 
Then   the desired result for \rethm{mainresult}~(\ref{item:c}) for the cases $n=3$ and $n=4$ is obtained  in Corollaries~\ref{cor:vb3} and~\ref{cor:excep}, respectively.
 
\medskip

We end this paper with an appendix explaining how the techniques of this paper can also be used to treat other braid-like groups such as the virtual twin groups.

\subsection*{Acknowledgments}

The first author was supported by Methusalem grant METH/21/03 -- long term structural funding of the Flemish Government.
The second author was partially supported by the National Council for Scientific and Technological Development - CNPq through a \textit{Bolsa de Produtividade} 305223/2022-4  and by   Projeto Tem\'atico-FAPESP Topologia Alg\'ebrica, Geom\'etrica e Diferencial 2016/24707-4 (Brazil).
The third author was partially supported by the National Council for Scientific and Technological Development - CNPq through a \textit{Bolsa de Produtividade} 305422/2022-7 and by Capes/Programa Capes-Print/ Processo n\'umero 88887.835402/2023-00.

\section{Preliminaries}\label{sec:prelim}

In this section we give the definitions of virtual braid groups that we use in the text and we 
prove \rethm{mainchar}.

\subsection{Virtual braid groups and forbidden relations}\label{subsec:vbn}

In this subsection, we recall the basic definitions of virtual braid groups.  
First, we write a presentation of the virtual braid group $VB_n$ that will be very useful in this work. 
This presentation appears in \cite{BB} and it is a reformulation of the one given in \cite[p.798]{V}.

\begin{defn}[{\cite[Theorem~4]{BB}}]\label{apvbn}
Let $n\geq 2$. The \emph{virtual braid group on $n$ strings}, denoted by $VB_n$, is the abstract group generated by $\sigma_i$ (classical generators) and $v_i$ (virtual generators), for $i=1,2,\dots,n-1$, with relations: \\
\hspace*{15mm}\parbox{10cm}{
\begin{itemize} 
       \item[(AR1)] $\sigma_i\sigma_{i+1}\sigma_{i}=\sigma_{i+1}\sigma_{i}\sigma_{i+1}$, $i=1,2,\dots,n-2;$
       \item[(AR2)] $\sigma_{i}\sigma_j=\sigma_{j}\sigma_i$, $\mid i-j\mid\ge 2$;
       \item[(PR1)] $v_iv_{i+1}v_{i}=v_{i+1}v_{i}v_{i+1}$, $i=1,2,\dots,n-2;$
       \item[(PR2)]  $v_{i}v_j=v_{j}v_i$, $\mid i-j\mid\ge 2;$
			 \item[(PR3)] $v_i^2=1$, $i=1,2,\dots,n-1;$
			 \item[(MR1)]  $\sigma_{i}v_j=v_{j}\sigma_i$, $ \mid i-j\mid\ge 2;$
			 \item[(MR2)] $v_iv_{i+1}\sigma_{i}=\sigma_{i+1}v_{i}v_{i+1}$, $i=1,2,\dots,n-2.$
\end{itemize}}
\end{defn}

\begin{rem}
    The letters AR, PR and MR that appear in Definition~\ref{apvbn} are used to indicate the type of relations in the given presentation of $VB_n$: \textit{Artin Relations, Permutation Relations and Mixed Relations}. 
\end{rem}

Let $n\geq 2$. As in \cite[Section~2]{BP}, from the presentation of $VB_n$ one can see that there are surjective homomorphisms $\pi_P\colon VB_n\to S_n$ and $\pi_K\colon VB_n\to S_n$ defined by $\pi_P(\sigma_i)=\pi_P(v_i)=\tau_i=(i,\, i+1)$ for all $1\leq i\leq n-1$ and by $\pi_K(\sigma_i)=1$ and $\pi_K(v_i)=\tau_i= (i,\, i+1)$ for all $1\leq i\leq n-1$, respectively. 
The kernel of $\pi_P$ is called the \textit{virtual pure braid group} and it is denoted by $VP_n$. 
A presentation of this group can be found in \cite[Theorem~1]{B}. 
The kernel of $\pi_K$, denoted by $KB_n$, is known to be an Artin group (this follows from the presentation of $KB_n$ given in \cite[Proposition~17]{BB}).  This fact has been useful to obtain properties of the virtual braid group itself, for instance it was essential in determining the set of endomorphisms of $VB_n$ in \cite{BP}. 
As mentioned in the first paragraph of Section~3 of \cite{B} (resp.\  in \cite[Section~6]{BB}) the virtual braid group admits a decomposition as semi-direct product $VB_n=VP_n\rtimes S_n$ (resp.\ $VB_n=KB_n\rtimes S_n$), with $\iota\colon S_n\longrightarrow VB_n$ given by $\iota(\tau_i)=v_i$, for $i=1,\dots,n-1$, being a section for $\pi_P$ (resp. for $\pi_K$). 

\begin{defn}\label{def:uvbn}
Consider, for $i=1,\dots,n-2$, the following so-called forbidden relations in the virtual braid group:
\begin{enumerate}
	\item \label{forb} $v_i\sigma_{i+1}\sigma_i = \sigma_{i+1}\sigma_iv_{i+1}$, 
	
	\item \label{forb2}  $v_{i+1}\sigma_i \sigma_{i+1} = \sigma_i\sigma_{i+1}v_i$.
\end{enumerate}
The \emph{welded braid group}, denoted by $WB_n$, is the quotient of $VB_n$ by the normal closure of the relations \textit{(\ref{forb})}.
The \emph{unrestricted virtual braid group}, denoted by $UVB_n$, is the quotient of $VB_n$ by the normal closure of the relations \textit{(\ref{forb})} and \textit{(\ref{forb2})}.
\end{defn}

\begin{rem}

We note that the welded braid group $WB_n$ appears with other names in the literature, for example as \emph{the loop braid group}, see \cite{D}. Also in \cite{D} one can find an extensive exposition of it. 
\end{rem}

Since the forbidden relations are preserved by $\pi_P\colon VB_n\to S_n$, we may define the homomorphisms $\overline{\pi_P}\colon WB_n\to S_n$
and $\overline{\overline{\pi_P}}\colon UVB_n\to S_n$ by $\overline{\pi_P}(\sigma_i)=\overline{\pi_P}(v_i)=(i,\, i+1)$ for all $1\leq i\leq n-1$ and $\overline{\overline{\pi_P}}(\sigma_i)=\overline{\overline{\pi_P}}(v_i)=(i,\, i+1)$ for all $1\leq i\leq n-1$, respectively. 
The kernel of $\overline{\pi_P}$ is called the \textit{welded pure braid group} and it is denoted by $WP_n$.
The kernel of $\overline{\overline{\pi_P}}$ is called the \textit{unrestricted virtual pure braid group} and it is denoted by $UVP_n$.

We note that, for $n\geq 3$, is not possible to define a similar homomorphism $\pi_K\colon VB_n\to S_n$ for the groups $WB_n$ and $UVB_n$ since the forbidden relations are not preserved by $\pi_K$.

\subsection{On characteristic subgroups}\label{subsec:char}

In this subsection we prove a general result about characteristic subgroups of a group. 
Recall that two homomorphisms of groups $\psi_1,\, \psi_2\colon G\to H$ are conjugate, denoted by $\psi_1\sim_{c} \psi_2$, if there exists $\eta\in \inn{H}$ such that $\psi_2=\eta\circ \psi_1$. The following lemma is an easy observation.

\begin{lem}
\label{lem:eqrel}
Let $\Sigma$ be the set of all homomorphisms from $G$ to $H$. 
Then $\sim_{c}$ is an equivalence relation on $\Sigma$.
\end{lem}

We note that \relem{eqrel} also holds if we restrict $\Sigma$ to the set of all surjective homomorphisms from $G$ to $H$.

\begin{proof}[Proof of \rethm{mainchar}]

Let $\lambda\in \Lambda$ such that for all $\omega \in \Lambda$ it holds that $Ker(\lambda) = Ker(\omega)$ as soon as they are isomorphic. Let $\phi\colon G\to G$ be any automorphism of $G$. 
Since $\lambda\circ \phi \in \Sigma$ then there is $\zeta\in \Lambda$ such that $\lambda\circ \phi\sim_{c} \zeta$, i.e.\ there exists an inner automorphism of $Q$, say $\eta$, such that $\lambda\circ \phi = \eta\circ \zeta$ and the following square is commutative
$$
\xymatrix{
G \ar^-{\zeta}[r] \ar_-{\phi}^{\cong}[d] & Q \ar[d]_- {\eta}^{\cong}\\
G \ar^-{\lambda}[r] & Q }
$$
It is easy to see that $\varphi(Ker(\zeta))=Ker(\lambda)$ and so 
$Ker(\zeta)$ is isomorphic to $Ker(\lambda)$. By the hypothesis, it follows that $Ker(\zeta) = Ker(\lambda)$, hence
$\varphi(Ker(\lambda))=\varphi(Ker(\zeta))= Ker(\lambda)$.

So we proved that $\phi(Ker(\lambda))=Ker(\lambda)$ for any automorphism of $G$, showing that $Ker(\lambda)$ is a characteristic subgroup of $G$.
\end{proof}

\begin{rem}
The proof of \rethm{mainchar} was in part motivated from the one given in \cite{M} to prove that, for $n\geq 5$, $UVP_n$ is a characteristic subgroup of $UVB_n$, see \cite[Proposition~2.15]{M}.   
\end{rem}

We finish this subsection with the following property about homomorphisms being conjugate.

\begin{prop}\label{prop:homconj}
Let $\overline{G}$ be a quotient of a group $G$ and $p\colon G\to \overline{G}$ the natural projection. 
Consider two homomorphisms $\overline{\zeta}_1, \overline{\zeta}_2\colon \overline{G}\to H$ and define $\zeta_1, \zeta_2\colon G\to H$ by $\zeta_i= \overline{\zeta}_i\circ p$, for $i=1,2$. 
Then, $\zeta_1\sim_c \zeta_2$ if and only if $\overline{\zeta}_1\sim_c \overline{\zeta}_2$. 
\end{prop}

\begin{proof}
Suppose that $\zeta_1\sim_c \zeta_2$. 
By definition there is $h\in H$ such that $\zeta_1(g) = h \zeta_2(g)h^{-1}$, for all $g\in G$.
So,  $\overline{\zeta}_1\circ p(g) = h \overline{\zeta}_2\circ p(g) h^{-1}$, for all $g\in G$, that is equivalent to $\overline{\zeta}_1 (\overline{g}) = h \overline{\zeta}_2(\overline{g}) h^{-1}$, for all $\overline{g}\in \overline{G}$. Hence, $\overline{\zeta}_1\sim_c \overline{\zeta}_2$.

The proof of the converse is similar.
\end{proof}

\section{Characteristic subgroups of virtual braid groups}\label{sec:char}

The main objective of this section is to prove that the pure virtual braid groups $VP_n$ are characteristic subgroups of $VB_n$ for $n\geq 4$ and that the groups $KB_n$ are characteristic subgroups of $VB_n$ for $n\geq 3$. We recall that a group homomorphism $\psi\colon G\to H$ is said to be \emph{abelian} if its image $\psi(G)$ is an abelian subgroup of $H$. 
The homomorphisms $\pi_P$ and $\pi_K$ defined in \resubsec{vbn} will be mentioned several times in this section.

\subsection{Two strings case}

From the presentation of the virtual braid group $VB_n$ and the definition of the groups $WB_n$ and $UVB_n$, see \resec{prelim}, it follows that
$$
VB_2=WB_2=UVB_2=\setang{\sigma_1, v_1}{v_1^2=1} \cong \Z\ast \Z_2. 
$$

\begin{rem}\label{rem:notchar}
 In \cite[Remark~2.6]{M} the author constructed an automorphism $\alpha\colon VB_2\to VB_2$ defined by $\alpha(\sigma_1)=\sigma_1^{-1}v_1$ and $\alpha(v_1)=v_1$. 
Using this automorphism she proved that $VP_2=WP_2=UVP_2$ is not characteristic in $VB_2=WB_2=UVB_2$. 

We note that the same automorphism may be used to verify that $KB_2$ is  not characteristic in $VB_2$.
\end{rem}

\subsection{Homomorphisms from virtual braid groups to the symmetric group: Three strings case}

In this subsection we will use the following presentation of the virtual braid group with 3 strings (see \resec{prelim} for a presentation of $VB_n$):
\begin{equation}\label{eq:presvb3}
VB_3=\setang{\sigma_1, \sigma_2, v_1, v_2}{\sigma_1\sigma_2\sigma_1=\sigma_2\sigma_1\sigma_2, v_1v_2v_1=v_2v_1v_2, v_1v_2\sigma_1=\sigma_2v_1v_2, v_1^2=1, v_2^2=1} 
\end{equation}
and the presentation
\begin{equation}\label{eqn:press3}
S_3=\setang{\tau_1, \tau_2}{\tau_1\tau_2\tau_1=\tau_2\tau_1\tau_2, \tau_1^2=1, \tau_2^2=1} 
\end{equation}
of the symmetric group $S_3$.

Define, for $1\leq i\leq 8$, the following homomorphisms $\psi_i\colon VB_3\to S_3$:
\begin{enumerate}
	\item $\psi_1(v_1)=\tau_1$, $\psi_1(v_2)=\tau_1$, $\psi_1(\sigma_1)=\tau_2$, $\psi_1(\sigma_2)=\tau_2$;
	\item $\psi_2(v_1)=\tau_1$, $\psi_2(v_2)=\tau_2$, $\psi_2(\sigma_1)=\tau_1$, $\psi_2(\sigma_2)=\tau_2$. 
	In this case $\psi_2$ is equal to the homomorphism $\pi_P$;  %defined in \cite[p.~1343]{BP};
  \item $\psi_3(v_1)=\tau_1$, $\psi_3(v_2)=\tau_2$, $\psi_3(\sigma_1)=\tau_2$, $\psi_3(\sigma_2)=\tau_1\tau_2\tau_1$;
  \item $\psi_4(v_1)=\tau_1$, $\psi_4(v_2)=\tau_2$, $\psi_4(\sigma_1)=\tau_1\tau_2\tau_1$, $\psi_4(\sigma_2)=\tau_1$;
  \item $\psi_5(v_1)=\tau_1$, $\psi_5(v_2)=\tau_2$, $\psi_5(\sigma_1)=\tau_1\tau_2$, $\psi_5(\sigma_2)=\tau_1\tau_2$;
  \item $\psi_6(v_1)=\tau_1$, $\psi_6(v_2)=\tau_2$, $\psi_6(\sigma_1)=\tau_2\tau_1$, $\psi_6(\sigma_2)=\tau_2\tau_1$;	
	\item $\psi_7(v_1)=\tau_1$, $\psi_7(v_2)=\tau_2$, $\psi_7(\sigma_1)=1$, $\psi_7(\sigma_2)=1$. In this case $\psi_7$ is equal to the homomorphism $\pi_K$; 
	%defined in \cite[p.~1343]{BP};
	\item $\psi_8(v_1)=\tau_1$, $\psi_8(v_2)=\tau_1$, $\psi_8(\sigma_1)=\tau_1\tau_2$, $\psi_8(\sigma_2)=\tau_1\tau_2$.
\end{enumerate}

\begin{rem}
We note that the homomorphisms $\{\psi_i\mid 1\leq i \leq 8\}$ are not pairwise conjugate. 
Indeed, first we can see that the homomorphisms $\psi_1$ and $\psi_8$ are not conjugate to $\psi_i$ for $i=2,3,\ldots, 7$ because 
$\psi_j(v_1)=\psi_j(v_2)$ for $j=1,8$ and this is not the case for the $\psi_i$'s with $i=2,3,\ldots, 7$. Moreover, $\psi_1$ is not conjugate to $\psi_8$ because $\psi_1(\sigma_1)$ is an element of order 2, while $\psi_8(\sigma_1)$ has order 3.

Now assume that $\psi_i$ is conjugate to $\psi_j$ with $3 \leq i,j\leq 7$ and let $\lambda\in S_3$ be such that $\psi_i = c_\lambda \circ \psi_j$.
As $\psi_i(v_1)=\psi_j(v_1) = \tau_1$, it follows that $\lambda \tau_1 \lambda^{-1} =\tau_1 $ and analogously from $\psi_i(v_2)=\psi_j(v_2) = \tau_2$ we find that $\lambda \tau_2 \lambda^{-1} =\tau_2$. Hence $\lambda$ centralizes both $\tau_1$ and $\tau_2$ and so $\lambda$ belongs to the center of $S_3$ (because $S_3$ is generated by $\tau_1$ and $\tau_2$). But the center of $S_3$ is trivial, so $\lambda =1$ from which we get that $\psi_i = \psi_j$ and so $i=j$. 
\end{rem}

\begin{thm}\label{thm:surjvb3}
Let $\psi\colon VB_3\to S_3$ be a homomorphism.  Then, up to conjugation, one of the following possibilities holds
\begin{enumerate}
	\item $\psi$ is abelian;
	\item $\psi\in \{\psi_i\mid 1\leq i \leq 8\}$.
\end{enumerate}
\end{thm}

\begin{proof}
 Let $\psi\colon VB_3\to S_3$ be a homomorphism and let $\iota\colon S_3\to VB_3$ be the natural inclusion defined by $\iota(\tau_1)=v_1$ and $\iota(\tau_2)=v_2$. 
Then, $\psi\circ \iota$ is an endomorphism of $S_3$.  It is straightforward  to verify that, up to conjugation, $\psi\circ \iota$ is either the identity homomorphism or  $im(\psi\circ \iota)=\langle\tau_1\rangle$ or  $im(\psi\circ \iota)=\{1\}$ (the trivial subgroup).   To say that   $\psi\circ \iota$ is abelian is equivalent to saying that   $\psi\circ \iota$ is, up to equivalence,  not the identity. 

\medskip

First suppose that $\psi\circ \iota$ is not the identity. 
Hence, from the relations $v_1^2=1$, $v_2^2=1$ and $v_1v_2v_1=v_2v_1v_2$ we have $\psi(v_1)=\psi(v_2)=w_1\in S_3$ with $w_1^2=1$. 
From the relation $v_1v_2\sigma_1=\sigma_2v_1v_2$ we get $\psi(\sigma_1)=\psi(\sigma_2)=w_2\in S_3$. 
Notice that, up to conjugacy, $w_1$ is either $1$ or $\tau_1$. 
Also if $w_1=1$ or $w_2=1$ then $\psi$ is abelian. 
    
So, we suppose that $w_1=\tau_1$ and that $w_2$ is a  non trivial element in $S_3$. 
Now we analyse the possible values of $ w_2 \in S_3$ such that $\psi$ is a homomorphism.
\begin{itemize}
	\item If $\psi(v_1)=\tau_1$, $\psi(v_2)=\tau_1$, $\psi(\sigma_1)=\tau_1$, $\psi(\sigma_2)=\tau_1$ then $\psi$ is abelian.
	\item Suppose that $\psi(v_1)=\tau_1$, $\psi(v_2)=\tau_1$, $\psi(\sigma_1)=\tau_2$, $\psi(\sigma_2)=\tau_2$. This homomorphism is $\psi_1$. 
	\item If $\psi(v_1)=\tau_1$, $\psi(v_2)=\tau_1$, $\psi(\sigma_1)=\tau_1\tau_2$, $\psi(\sigma_2)=\tau_1\tau_2$ This homomorphism is $\psi_8$. 
	\item Suppose that $\psi(v_1)=\tau_1$, $\psi(v_2)=\tau_1$, $\psi(\sigma_1)=\tau_1\tau_2\tau_1$, $\psi(\sigma_2)=\tau_1\tau_2\tau_1$. 
	Then $\psi$ is conjugate to $\psi_1$.
	\item Let $\psi(v_1)=\tau_1$, $\psi(v_2)=\tau_1$, $\psi(\sigma_1)=\tau_2\tau_1$, $\psi(\sigma_2)=\tau_2\tau_1$. 
	Then $\psi$ is conjugate to $\psi_8$.

 \end{itemize}
For any choice of $w_2=\psi(\sigma_1)=\psi(\sigma_2)$ in $S_3$ we obtain an abelian homomorphism or a homomorphism that is conjugate to $\psi_1$ or $\psi_8$.

\medskip

Now, suppose that $\psi\circ \iota$ is the identity homomorphism. 
This implies that 
$$
\psi(v_1)=\tau_1 \textrm{ and } \psi(v_2)=\tau_2.
$$
From the mixed relation $v_1v_2\sigma_1v_2v_1=\sigma_2$ it follows that  if we know $\psi(\sigma_1)$ then $\psi(\sigma_2)$ is completely determined. 
We analyse the possible values of $\psi(\sigma_1)$.
\begin{itemize}
	\item Suppose that $\psi(\sigma_1)=1$, $\psi(\sigma_2)=1$. This homomorphism is  $\psi_7=\pi_K$.
	%defined in \cite[p.~1343]{BP}.
	\item Suppose that $\psi(\sigma_1)=\tau_1$. Then $\psi(\sigma_2)=\tau_1\tau_2\tau_1\tau_2\tau_1=\tau_2$. This homomorphism is  $\psi_2=\pi_P$.
	%defined in \cite[p.~1343]{BP}.
  \item Let $\psi(\sigma_1)=\tau_2$, then $\psi(\sigma_2)=\tau_1\tau_2\tau_2\tau_2\tau_1=\tau_1\tau_2\tau_1$. This homomorphism is $\psi_3$.
  \item If $\psi(\sigma_1)=\tau_1\tau_2\tau_1$, $\psi(\sigma_2)= \tau_1\tau_2 \tau_2\tau_1\tau_2 \tau_2\tau_1 =\tau_1$. This homomorphism is $\psi_4$.
  \item Let $\psi(\sigma_1)=\tau_1\tau_2$. So, $\psi(\sigma_2)= \tau_1\tau_2 \tau_1\tau_2 \tau_2\tau_1 =\tau_1\tau_2$. This   homomorphism is $\psi_5$.
  \item Suppose that $\psi(\sigma_1)=\tau_2\tau_1$. Then $\psi(\sigma_2)= \tau_1\tau_2 \tau_2\tau_1 \tau_2\tau_1 =\tau_2\tau_1$. This  homomorphism is $\psi_6$.
\end{itemize}
From the computations above we proved that, up to conjugation, $\psi$ is abelian or $\psi\in \{\psi_i\mid 1\leq i \leq 8\}$.
\end{proof}

Now we study the homomorphisms from the welded braid group $WB_3$ and the unrestricted braid group $UVB_3$ to the symmetric group $S_3$. 
The study of these homomorphisms follows the same lines as the proof of \rethm{surjvb3}. 
Instead of repeating this method, we may verify which homomorphisms given in \rethm{surjvb3} respect the forbidden relations given in \redef{uvbn}. 

\begin{thm}\label{thm:surjwb3}
Let $\omega\colon WB_3\to S_3$ be a homomorphism.  
Then, up to conjugation, one of the following possibilities holds
\begin{enumerate}
	\item $\omega$ is abelian;
	\item $\omega=\overline{\psi}$, where $\psi\in \{\psi_i\mid 1\leq i \leq 5\}$ as in \rethm{surjvb3} and $\overline{\psi}\colon WB_3\to S_3$ is the induced homomorphism in the quotient of $VB_3$ by adding the forbidden relation \textit{(\ref{forb})} of \redef{uvbn}.
\end{enumerate}
\end{thm}

\begin{proof}
Let $\omega\colon WB_3\to S_3$ be a homomorphism. If $\omega$ is abelian there is nothing to prove. Suppose that $\omega$ is non-abelian. 
Let $\psi\colon VB_3\to S_3$ be a homomorphism that belongs to $\{\psi_i\mid 1\leq i \leq 8\}$ as in \rethm{surjvb3}. 
We verify which of these homomorphisms satisfy the forbidden relation \textit{(\ref{forb})} of \redef{uvbn}.

Let $\psi=\psi_1$. 
Then $\psi_1(v_1\sigma_2\sigma_1)=\tau_1\tau_2\tau_2=\tau_1$ and $\psi_1(\sigma_2\sigma_1v_2)=\tau_2\tau_2\tau_1=\tau_1$. Hence, $\overline{\psi_1}\colon WB_3\to S_3$ is a homomorphism. 
Now, let $\psi=\psi_7$. 
Recall that $\psi_7(v_1)=\tau_1$, $\psi_7(v_2)=\tau_2$, $\psi_7(\sigma_1)=1$ and $\psi_7(\sigma_2)=1$.
Then $\psi_7(v_1\sigma_2\sigma_1)=\tau_1$ and $\psi_7(\sigma_2\sigma_1v_2)=\tau_2$. Hence, $\psi_7$ does not induce a homomorphism in the quotient group $WB_3$. The verification of the other homomorphisms $\psi\in \{\psi_i\mid 1\leq i \leq 8\}$ are similar and with this computation we obtain the result of this theorem.
\end{proof}

Using the same idea of the proof of the last theorem we get the following result about unrestricted virtual braid groups.

\begin{thm}\label{thm:surjuvb3}
Let $\mu\colon UVB_3\to S_3$ be a homomorphism. 
Then, up to conjugation, one of the following possibilities holds
\begin{enumerate}
	\item $\mu$ is abelian;
	\item $\mu=\overline{\overline{\psi}}$, where $\psi\in \{\psi_i\mid 1\leq i \leq 4\}$ as in \rethm{surjvb3} and $\overline{\overline{\psi}}\colon UVB_3\to S_3$ is the induced homomorphism in the quotient of $VB_3$ by adding the forbidden relations \textit{(\ref{forb})} and \textit{(\ref{forb2})} of \redef{uvbn}.
\end{enumerate}
\end{thm}

\begin{proof}
The proof of this theorem is similar to the one of \rethm{surjwb3}.
\end{proof}

\begin{rem}
It follows from \reprop{homconj} and \rethm{surjvb3} that the homomorphisms $\{\overline{\psi_i}\mid 1\leq i \leq 5\}$ of \rethm{surjwb3}(b) are pairwise non conjugate. 
Similarly for the homomomorphisms of \rethm{surjuvb3}(b). 

\end{rem}

\subsection{Homomorphisms from virtual braid groups to the symmetric group: Four strings case}

In this subsection we will use the following presentation of the symmetric group
\begin{equation}\label{eqn:press3}
S_4=\setang{\tau_1, \tau_2, \tau_3}{\tau_1\tau_2\tau_1=\tau_2\tau_1\tau_2, \tau_2\tau_3\tau_2=\tau_3\tau_2\tau_3, \tau_1\tau_3 = \tau_3\tau_1, \tau_1^2=\tau_2^2=\tau_3^2=1} 
\end{equation}
and the presentation of the virtual braid group with 4 strings (see \resec{prelim} for a presentation of $VB_n$) with generators $\sigma_1, \sigma_2, \sigma_3, v_1, v_2, v_3$ and defining relations:
\begin{itemize}
	\item[(AR)] $\sigma_1\sigma_2\sigma_1=\sigma_2\sigma_1\sigma_2$, $\sigma_2\sigma_3\sigma_2=\sigma_3\sigma_2\sigma_3$, $\sigma_1\sigma_3=\sigma_3\sigma_1$, 
	
	\item[(PR)] $v_1v_2v_1=v_2v_1v_2$, $v_2v_3v_2=v_3v_2v_3$, $v_1v_3=v_3v_1$, $v_1^2=1$, $v_2^2=1$, $v_3^2=1$,

  \item[(MR)] $\sigma_1v_3=v_3\sigma_1$, $\sigma_3v_1=v_1\sigma_3$, $v_1v_2\sigma_1=\sigma_2v_1v_2$, $v_2v_3\sigma_2=\sigma_3v_2v_3$.
\end{itemize}

Let $\eta\colon S_4\to S_4$ be the homomorphism defined by $\eta(\tau_1)=\eta(\tau_3)=\tau_1$ and $\eta(\tau_2)=\tau_2$.

\begin{lem}\label{lem:lems4}

Let $\phi\colon S_4\to S_4$ be any endomorphism of $S_4$. Then, up to conjugation, one of the following possibilities holds.
\begin{enumerate}
	\item $\phi$ is abelian,
	\item $\phi$ is the identity homomorphism,
	\item $\phi=\eta$.
\end{enumerate}
\end{lem}

\begin{proof}
This follows from examining case by case all the possible images of the transpositions. 
We note that if, for some $i=1,2,3$, $\phi(\tau_i)=1$ then $\phi$ is the trivial homomorphism. 
Hence, we do not consider this case.
Also, we notice that if, for some $i=1,2,3$, $\phi(\tau_i)$ is a product of different transpositions then $\phi$ is abelian with image the cyclic group of order $2$.

Suppose that $\phi(\tau_1)=\tau_1$. 
It follows from the relation $\tau_1\tau_3 = \tau_3\tau_1$ that $\phi(\tau_3)$ lies in $\{1, \tau_1, \tau_3, \tau_1\tau_3 \},$ the centralizer of $\langle \tau_1 \rangle$ in $S_4$. 
By examining the possible images of $\tau_3$ we obtain, up to conjugation, that either $\phi$ is abelian, or $\phi$ is the identity homomorphism or $\phi=\eta$.
\end{proof}

Define, for $1\leq i\leq 6$, the following homomorphisms $\delta_i\colon VB_4\to S_4$:
\begin{enumerate}
	\item $\delta_1(v_1)=\tau_1$, $\delta_1(v_2)=\tau_2$, $\delta_1(v_3)=\tau_1$, $\delta_1(\sigma_1)=\tau_1$, $\delta_1(\sigma_2)=\tau_2$, $\delta_1(\sigma_3)=\tau_1$;
	\item $\delta_2(v_1)=\tau_1$, $\delta_2(v_2)=\tau_2$, $\delta_2(v_3)=\tau_1$, $\delta_2(\sigma_1)=\tau_3$, $\delta_2(\sigma_2)=\tau_3\tau_2\tau_1\tau_2\tau_3$, $\delta_2(\sigma_3)=\tau_3$;
  \item $\delta_3(v_1)=\tau_1$, $\delta_3(v_2)=\tau_2$, $\delta_3(v_3)=\tau_3$, $\delta_3(\sigma_1)=\tau_1$, $\delta_3(\sigma_2)=\tau_2$, $\delta_3(\sigma_3)=\tau_3$. 
	This $\delta_3$ is equal to the homomorphism $\pi_P$; 
	%defined in \cite[p.~1343]{BP};
  \item $\delta_4(v_1)=\tau_1$, $\delta_4(v_2)=\tau_2$, $\delta_4(v_3)=\tau_3$, $\delta_4(\sigma_1)=\tau_3$, $\delta_4(\sigma_2)=\tau_3\tau_2\tau_1\tau_2\tau_3$, $\delta_4(\sigma_3)=\tau_1$;
	\item $\delta_5(v_1)=\tau_1$, $\delta_5(v_2)=\tau_2$, $\delta_5(v_3)=\tau_3$, $\delta_5(\sigma_1)=1$, $\delta_5(\sigma_2)=1$, $\delta_5(\sigma_3)=1$. This $\delta_5$ is equal to the homomorphism $\pi_K$; 
	%defined in \cite[p.~1343]{BP};
	\item $\delta_6(v_1)=\tau_1$, $\delta_6(v_2)=\tau_2$, $\delta_6(v_3)=\tau_1$, $\delta_6(\sigma_1)=1$, $\delta_6(\sigma_2)=1$, $\delta_6(\sigma_3)=1$.
\end{enumerate}

\begin{rem}
We note that the homomorphisms $\{\delta_i\mid 1\leq i \leq 6\}$ are not pairwise conjugate.
\end{rem}

\begin{thm}\label{thm:surjvb4}
Let $\delta\colon VB_4\to S_4$ be a homomorphism.  Then, up to conjugation, one of the following possibilities holds
\begin{enumerate}
	\item $\delta$ is abelian;
	\item $\delta\in \{\delta_i\mid 1\leq i \leq 6\}$.
\end{enumerate}
\end{thm}

\begin{proof}
Let $\delta\colon VB_4\to S_4$ be a homomorphism and let $\iota\colon S_4\to VB_4$ be the natural inclusion defined by $\iota(\tau_1)=v_1$, $\iota(\tau_2)=v_2$ and $\iota(\tau_3)=v_3$. 
Then, $\delta\circ \iota$ is an endomorphism of $S_4$. 
From \relem{lems4}, up to conjugation, $\delta\circ \iota$ is abelian or it is $\eta$ or it is the identity homomorphism. 

We claim that when $\delta\circ \iota$ is abelian then $\delta$ is abelian. The proof given for \cite[Theorem~2.1]{BP} in the case in which the composition is abelian works for $n=4$, proving our claim.

Suppose that $\delta\circ \iota=\eta$. 
Therefore 
$$
\delta(v_1)=\tau_1,  \delta(v_2)=\tau_2 \textrm{ and } \delta(v_3)=\tau_1.
$$
From the mixed relations $v_1v_2\sigma_1v_2v_1=\sigma_2$ and $v_2v_3\sigma_2v_3v_2=\sigma_3$ we see that $\delta(\sigma_2)$ and $\delta(\sigma_3)$ are completely determined by $\delta(\sigma_1)$.
We analyse the possible values of $\delta(\sigma_1)$. 
From the relation $\sigma_1v_3 = v_3\sigma_1$ and $\delta(v_3)=\tau_1$ it follows that $\delta(\sigma_1)$ lies in the centralizer of $\langle \tau_1 \rangle$ in $S_4$, i.e. $\delta(\sigma_1) \in \{1, \tau_1, \tau_3, \tau_1\tau_3 \}$. 
\begin{itemize}
	\item Suppose that $\delta(\sigma_1)=1$, then $\delta(\sigma_2)=1$ and $\delta(\sigma_3)=1$. This homomorphism is $\delta_6$. 
	
	\item Suppose that $\delta(\sigma_1)=\tau_1$. Then $\delta(\sigma_2)=\tau_1\tau_2\tau_1\tau_2\tau_1=\tau_2$ and $\delta(\sigma_3)=\tau_2\tau_1\tau_2\tau_1\tau_2=\tau_1$. This homomorphism is $\delta_1$.  

	\item Suppose that $\delta(\sigma_1)=\tau_3$. Then $\delta(\sigma_2)=\tau_1\tau_2\tau_3\tau_2\tau_1=\tau_3\tau_2\tau_1\tau_2\tau_3$ and $\delta(\sigma_3)=\tau_2\tau_1\tau_1\tau_2\tau_3\tau_2\tau_1\tau_1\tau_2=\tau_3$. This homomorphism is $\delta_2$.  

	\item Finally, if $\delta(\sigma_1)=\tau_1\tau_3$ we do not get a homomorphism since the relation $\sigma_1\sigma_2\sigma_1=\sigma_2\sigma_1\sigma_2$ is not preserved.
\end{itemize}

Now, suppose that $\delta\circ \iota$ is the identity homomorphism. 
This implies that 
$$
\delta(v_1)=\tau_1,  \delta(v_2)=\tau_2 \textrm{ and } \delta(v_3)=\tau_3.
$$
As before, from the mixed relations $v_1v_2\sigma_1v_2v_1=\sigma_2$ and $v_2v_3\sigma_2v_3v_2=\sigma_3$ it follows that  $\delta(\sigma_1)$ determines $\delta(\sigma_2)$ and $\delta(\sigma_3)$ completely. Moreover the relation $\sigma_1 v_3=v_3 \sigma_1$ implies that $\delta(\sigma_1) \in \{1, \tau_1, \tau_3, \tau_1\tau_3 \}$ the centralizer of $\langle \tau_3 \rangle$ in $S_4$. 
We analyse the possible values of $\delta(\sigma_1)$.
\begin{itemize}
	\item Suppose that $\delta(\sigma_1)=1$, then $\delta(\sigma_2)=1$ and $\delta(\sigma_3)=1$. This homomorphism is $\delta_5=\pi_K$.
	%defined in \cite[p.~1343]{BP}.

	\item Suppose that $\delta(\sigma_1)=\tau_1$. Then $\delta(\sigma_2)=\tau_1\tau_2\tau_1\tau_2\tau_1=\tau_2$ and $\delta(\sigma_3)=\tau_2\tau_3\tau_2\tau_3\tau_2=\tau_3$. This homomorphism is $\delta_3=\pi_P$. 
	%defined in \cite[p.~1343]{BP}.

	\item Suppose that $\delta(\sigma_1)=\tau_3$. Then $\delta(\sigma_2)=\tau_1\tau_2\tau_3\tau_2\tau_1=\tau_3\tau_2\tau_1\tau_2\tau_3$ and $\delta(\sigma_3)=\tau_2\tau_3\tau_3\tau_2\tau_1\tau_2\tau_3\tau_3\tau_2=\tau_1$. This homomorphism is $\delta_4$.  

	\item Finally, if $\delta(\sigma_1)=\tau_1\tau_3$ we do not get a homomorphism since the relation $\sigma_1\sigma_2\sigma_1=\sigma_2\sigma_1\sigma_2$ is not preserved.
\end{itemize}

From the computations above we proved that, up to conjugation, $\delta$ is abelian or $\delta\in \{\delta_i\mid 1\leq i \leq 6\}$.
\end{proof}

Similar to the case $n=3$, we verify which homomorphisms given in \rethm{surjvb4} respect the forbidden relations given in \redef{uvbn}. 

\begin{thm}\label{thm:surjwb4}
Let $\omega\colon WB_4\to S_4$ be a homomorphism.  
Then, up to conjugation, one of the following possibilities holds
\begin{enumerate}
	\item $\omega$ is abelian;
	\item $\omega=\overline{\delta}$, where $\delta\in \{\delta_i\mid 1\leq i \leq 4\}$ as in \rethm{surjvb4} and $\overline{\delta}\colon WB_4\to S_4$ is the induced homomorphism in the quotient of $VB_4$ by adding the forbidden relation \textit{(\ref{forb})} of \redef{uvbn}.
\end{enumerate}
\end{thm}

\begin{proof}
The proof is completely similar to the one given for \rethm{surjvb3}.
\end{proof}

Finally we get the result for unrestricted virtual braid groups.

\begin{thm}\label{thm:surjuvb4}
Let $\mu\colon UVB_4\to S_4$ be a homomorphism. 
Then, up to conjugation, one of the following possibilities holds
\begin{enumerate}
	\item $\mu$ is abelian;
	\item $\mu=\overline{\overline{\delta}}$, where $\delta\in \{\delta_i\mid 1\leq i \leq 4\}$ as in \rethm{surjvb4} and $\overline{\overline{\delta}}\colon UVB_4\to S_4$ is the induced homomorphism in the quotient of $VB_4$ by adding the forbidden relations \textit{(\ref{forb})} and \textit{(\ref{forb2})} of \redef{uvbn}.
\end{enumerate}
\end{thm}

\begin{proof}
The proof is completely similar to the one given for \rethm{surjvb3}.
\end{proof}

\begin{rem}
It follows from \reprop{homconj} and \rethm{surjvb4} that the homomorphisms $\{\overline{\delta_i}\mid 1\leq i \leq 4\}$ of \rethm{surjwb4}(b) are pairwise non conjugate. 
Similarly for the homomorphisms of \rethm{surjuvb4}(b). 

\end{rem}

\subsection{Proof of \rethm{maincharvbn} }
\label{subsec:thm2}
We will use \rethm{mainchar} to prove that for $n\geq 4$ (resp.\ $n\geq 3$) the pure subgroup of the virtual braid group (resp.\ the welded braid group and the unrestricted virtual braid group) is characteristic in the virtual braid group (resp.\ the welded braid group and the unrestricted virtual braid group). 
In the previous subsections we computed, up to conjugation, all surjective homomorphisms from $VB_n$ to $S_n$ (and also for $WB_n$ and $UVB_n$), for $n=3,4$. 
In the next two lemmas we will compare some kernels of these maps.

We will use the following notation. 
Let $G$ be a group, the abelianization of $G$ will be denoted by $G^{Ab}$, i.e. $G^{Ab}=G/[G,G]$. 

\begin{lem}\label{lem:notisovbn}
Let $n\geq 3$. 
\begin{enumerate}
	\item The groups $VP_n$ and $KB_n$ are not isomorphic. 

	\item Let $n=3$ and let $\psi\in \{ \psi_i \mid 1\leq i\leq 8\}$ as in \rethm{surjvb3}. 	
	\begin{itemize}

		\item The group $KB_3=Ker(\psi_7)$ is not isomorphic to $Ker(\psi_i)$ for $1\leq i\leq 8$ with $i\neq 7$.	
\end{itemize}

	\item Let $n=4$ and let $\delta\in \{ \delta_i \mid 1\leq i\leq 6\}$ as in \rethm{surjvb4}. 
	\begin{itemize}
		\item The group $VP_4=Ker(\delta_3)$ is not isomorphic to $Ker(\delta_i)$ for $1\leq i\leq 6$ with $i\neq 3$. 
		\item The group $KB_4=Ker(\delta_5)$ is not isomorphic to $Ker(\delta_i)$ for $1\leq i\leq 6$ with $i\neq 5$.	
\end{itemize}

\end{enumerate}
\end{lem}

\begin{proof}
\begin{enumerate}
	\item This item is the same as \cite[Proposition~21]{BB}.

\item Let $n=3$ and let $\psi\in \{ \psi_i \mid 1\leq i\leq 8\}$ as in \rethm{surjvb3}. 	
We used the GAP System \cite{GAP} to compute the abelianization of the groups involved. We elucidate the routine used in the computations for the case $n=3$:
\begin{lstlisting}[language=GAP]
 f4:=FreeGroup("x","y","a","b");;
 AssignGeneratorVariables(f4);;
 r:=ParseRelators([x,y,a,b],"xyx=yxy,aba=bab,a^2=1,b^2=1,bayab=x");;
 g:= f4/r; # g is the virtual braid group on 3 strings
 p1:=GroupHomomorphismByImages(g, SymmetricGroup(3), [g.1,g.2, g.3, g.4], 
[(2,3), (2,3), (1,2), (1,2)]); AbelianInvariants(Kernel(p1)); 
[ f1, f2, f3, f4 ] -> [ (2,3), (2,3), (1,2), (1,2) ]
[ 0, 0, 0, 0, 3, 3, 3 ]
 p2:=GroupHomomorphismByImages(g, SymmetricGroup(3), [g.1,g.2, g.3, g.4], 
[(1,2), (2,3), (1,2), (2,3)]); AbelianInvariants(Kernel(p2));
[ f1, f2, f3, f4 ] -> [ (1,2), (2,3), (1,2), (2,3) ]
[ 0, 0, 0, 0, 0, 0 ]
 p3:=GroupHomomorphismByImages(g, SymmetricGroup(3), [g.1,g.2, g.3, g.4], 
[(2,3), (1,2)*(2,3)*(1,2), (1,2), (2,3)]); AbelianInvariants(Kernel(p3));
[ f1, f2, f3, f4 ] -> [ (2,3), (1,3), (1,2), (2,3) ]
[ 0, 0, 0, 0, 0, 0 ]
 p4:=GroupHomomorphismByImages(g, SymmetricGroup(3), [g.1,g.2, g.3, g.4], 
[(1,2)*(2,3)*(1,2), (1,2), (1,2), (2,3)]); AbelianInvariants(Kernel(p4));
[ f1, f2, f3, f4 ] -> [ (1,3), (1,2), (1,2), (2,3) ]
[ 0, 0, 0, 0, 0, 0 ]
 p5:=GroupHomomorphismByImages(g, SymmetricGroup(3), [g.1,g.2, g.3, g.4], 
[(1,2)*(2,3), (1,2)*(2,3), (1,2), (2,3)]); AbelianInvariants(Kernel(p5));
[ f1, f2, f3, f4 ] -> [ (1,3,2), (1,3,2), (1,2), (2,3) ]
[ 0, 0, 2, 2, 2, 2 ]
 p6:=GroupHomomorphismByImages(g, SymmetricGroup(3), [g.1,g.2, g.3, g.4], 
[(2,3)*(1,2), (2,3)*(1,2), (1,2), (2,3)]); AbelianInvariants(Kernel(p6));
[ f1, f2, f3, f4 ] -> [ (1,2,3), (1,2,3), (1,2), (2,3) ]
[ 0, 0, 2, 2, 2, 2 ]
 p7:=GroupHomomorphismByImages(g, SymmetricGroup(3), [g.1,g.2, g.3, g.4], 
[(), (), (1,2), (2,3)]); AbelianInvariants(Kernel(p7));
[ f1, f2, f3, f4 ] -> [ (), (), (1,2), (2,3) ]
[ 0, 0 ]
 p8:=GroupHomomorphismByImages(g, SymmetricGroup(3), [g.1,g.2, g.3, g.4], 
[(1,2)*(2,3), (1,2)*(2,3), (1,2), (1,2)]);AbelianInvariants(Kernel(p8));
[ f1, f2, f3, f4 ] -> [ (1,3,2), (1,3,2), (1,2), (1,2) ]
[ 0, 0, 0, 0, 3 ]
\end{lstlisting}
Summarizing, we get
\begin{multicols}{2}
\begin{itemize}
	\item $(Ker(\psi_1))^{Ab}\cong \Z^4\oplus (\Z_3)^3$
	\item $(Ker(\psi_2))^{Ab}=(VP_3)^{Ab}\cong \Z^6$
	\item $(Ker(\psi_3))^{Ab}\cong \Z^6$
	\item $(Ker(\psi_4))^{Ab}\cong \Z^6$
	\item $(Ker(\psi_5))^{Ab}\cong \Z^2\oplus (\Z_2)^4$
	\item $(Ker(\psi_6))^{Ab}\cong \Z^2\oplus (\Z_2)^4$
	\item $(Ker(\psi_7))^{Ab}=(KB_3)^{Ab}\cong \Z^2$
	\item $(Ker(\psi_8))^{Ab}\cong \Z^4\oplus  \Z_3$
\end{itemize}
\end{multicols}
From this we conclude that the group $KB_3=Ker(\psi_7)$ is not isomorphic to $Ker(\psi_i)$ for $1\leq i\leq 8$ with $i\neq 7$.

	\item Let $n=4$ and let $\delta\in \{ \delta_i \mid 1\leq i\leq 6\}$ as in \rethm{surjvb4}. We use the same idea of the previous item. 
	From the computations using GAP we get
\begin{multicols}{2}
\begin{itemize}
	\item $(Ker(\delta_1))^{Ab}\cong \Z^3\oplus (\Z_2)^2$
	\item $(Ker(\delta_2))^{Ab}\cong \Z^6\oplus (\Z_2)^8$
	\item $(Ker(\delta_3))^{Ab}=(VP_4)^{Ab}\cong \Z^{12}$
	\item $(Ker(\delta_4))^{Ab}\cong \Z^6\oplus (\Z_2)^2$
	\item $(Ker(\delta_5))^{Ab}=(KB_4)^{Ab}\cong  \Z$
	\item $(Ker(\delta_6))^{Ab}\cong \Z\oplus (\Z_2)^2$
\end{itemize}
\end{multicols}
	From these computations we conclude that the group $VP_4=Ker(\delta_3)$ is not isomorphic to $Ker(\delta_i)$ for $1\leq i\leq 6$ with $i\neq 3$ and also that the group $KB_4=Ker(\delta_5)$ is not isomorphic to $Ker(\delta_i)$ for $1\leq i\leq 6$ with $i\neq 5$.	
\end{enumerate}
\end{proof}

In the next result we consider the cases of welded and unrestricted virtual braid groups with few strings.

\begin{lem}\label{lem:notisowbn}
\begin{enumerate}
	\item Let $n=3$ 	
	\begin{itemize}
		\item Let $\overline{\psi_i}\colon WB_3\to S_3$ as in \rethm{surjwb3}, where $\{ \psi_i \mid 1\leq i\leq 5\}$ are the homomorphisms given in \rethm{surjvb3}. 
			 The pure welded braid subgroup $WP_3=Ker(\overline{\psi_2})$ is not isomorphic to $Ker(\overline{\psi_i})$, for any $1\leq i\leq 5$ with $i\neq 2$.

 		 \item Let $\overline{\overline{\psi_i}}\colon UVB_3\to S_3$ as in \rethm{surjuvb3}, where $\{ \psi_i \mid 1\leq i\leq 4\}$ are the homomorphisms given in \rethm{surjvb3}. 
			 The pure unrestricted virtual braid subgroup $UVP_3=Ker(\overline{\overline{\psi_2}})$ is not isomorphic to $Ker(\overline{\overline{\psi_i}})$, for any $1\leq i\leq 4$ with $i\neq 2$. 

\end{itemize}

		\item Let $n=4$. 
	\begin{itemize}
		 \item Let $\overline{\delta_i}\colon WB_4\to S_4$ as in \rethm{surjwb4}, where $\{ \delta_i \mid 1\leq i\leq 4\}$ are the homomorphisms given in \rethm{surjvb4}. 
			 The pure welded braid subgroup $WP_4=Ker(\overline{\delta_3})$ is not isomorphic to $Ker(\overline{\delta_i})$, for any $1\leq i\leq 4$ with $i\neq 3$. 
			 
 		 \item Let $\overline{\overline{\delta_i}}\colon UVB_4\to S_4$ as in \rethm{surjuvb4}, where $\{ \delta_i \mid 1\leq i\leq 4\}$ are the homomorphisms given in \rethm{surjvb4}. 
			 The pure unrestricted virtual braid subgroup $UVP_4=Ker(\overline{\overline{\delta_3}})$ is not isomorphic to $Ker(\overline{\overline{\delta_i}})$, for any $1\leq i\leq 4$ with $i\neq 3$. 
	\end{itemize}

\end{enumerate}
\end{lem}

\begin{proof}
The proof of this result is similar to the previous one in which we use GAP \cite{GAP} to compute the abelizanization of the kernel of each homomorphism. 
We just list below the abelianizations of the groups involved from which we conclude this result.

\begin{enumerate}
	\item Let $n=3$
	\begin{multicols}{2}
\begin{itemize}
	\item $(Ker(\overline{\psi_1}))^{Ab}\cong \Z^2\oplus (\Z_3)^3$
	\item $(Ker(\overline{\psi_2}))^{Ab}\cong \Z^6$
	\item $(Ker(\overline{\psi_3}))^{Ab}\cong \Z^4$
	\item $(Ker(\overline{\psi_4}))^{Ab}\cong \Z^4$
	\item $(Ker(\overline{\psi_5}))^{Ab}\cong \Z\oplus (\Z_3)^5$
	\item $(Ker(\overline{\overline{\psi_1}}))^{Ab}\cong \Z^2\oplus (\Z_3)^2$
	\item $(Ker(\overline{\overline{\psi_2}}))^{Ab}\cong \Z^6$
	\item $(Ker(\overline{\overline{\psi_3}}))^{Ab}\cong \Z^2\oplus \Z_3$
	\item $(Ker(\overline{\overline{\psi_4}}))^{Ab}\cong \Z^2\oplus \Z_3$

\end{itemize}
\end{multicols}
	
	\item Let $n=4$
	\begin{multicols}{2}
\begin{itemize}
	\item $(Ker(\overline{\delta_1}))^{Ab}\cong \Z^3\oplus (\Z_2)^2$
	\item $(Ker(\overline{\delta_2}))^{Ab}\cong \Z^3\oplus (\Z_2)^8$
	\item $(Ker(\overline{\delta_3}))^{Ab}\cong \Z^{12}$
	\item $(Ker(\overline{\delta_4}))^{Ab}\cong \Z^3\oplus (\Z_2)^3$
	\item $(Ker(\overline{\overline{\delta_1}}))^{Ab}\cong \Z^3\oplus (\Z_2)^2$
	\item $(Ker(\overline{\overline{\delta_2}}))^{Ab}\cong \Z^3\oplus (\Z_2)^6$
	\item $(Ker(\overline{\overline{\delta_3}}))^{Ab}\cong \Z^{12}$
	\item $(Ker(\overline{\overline{\delta_4}}))^{Ab}\cong \Z^3\oplus (\Z_2)^3$

\end{itemize}
\end{multicols}
	
\end{enumerate}
\end{proof}

\begin{rem}\label{rem:aut}
Given a group homomorphism $\xi\colon G \to H$ and $\gamma\in \aut{H}$, then $\ker{\xi} = \ker{\gamma\circ \xi}$.
\end{rem}

With the above information we can now determine eaxctly when the virtual pure braid group is characteristic in the virtual braid group.

\begin{proof}[Proof of \rethm{maincharvbn}]
The case $n=2$ follows from \rerem{notchar}. \\
Let $n\geq 3$. 
We first show that $VP_3$ is not characteristic in $VB_3$. Let $\alpha\colon VB_3 \to VB_3$ be the homomorphism determined by 
$$\alpha(v_1) =v_1, \alpha(v_2) = v_2, \alpha(\sigma_1) = v_1 v_2 \sigma_1 v_2 v_1 = \sigma_2 \mbox{ and } 
\alpha( \sigma_2) = v_1v_2 \sigma_2 v_2v_1. $$
We leave it to the reader to check that this $\alpha$ preserves the relations of the presentation~\eqref{eq:presvb3} of $VB_3$ and so indeed determines a homomorphism. Moreover, as $v_1v_2$ is an element of order 3, we have that 
\[ \alpha^3 (v_i)= v_i \mbox{ and }  \alpha^3 (\sigma_i) =
v_1v_2v_1v_2v_1v_2 \sigma_i v_2v_1v_2v_1v_2v_1 = \sigma_i, \; i=1,2,\]
hence $\alpha^3$ is the identity on $VB_3$ from which we conclude that $\alpha$ is an automorphism of $VB_3$. Note that $\pi_P(v_1\sigma_1)=1$ so $v_1\sigma_1\in VP_3$, but 
$\pi_P(\alpha(v_1\sigma_1)) = \pi_P( v_2 \sigma_1 v_2 v_1) =
\tau_2 \tau_1 \tau_2 \tau_1 = \tau_1 \tau_2 \neq 1$ showing that 
$\alpha (VP_3) \neq VP_3$ and so $VP_3$ is not a characteristic subgroup of $VB_3$.

We will apply \rethm{mainchar} to show that in the other cases we do obtain characteristic subgroups.
Recall that $\out{S_n}$ is trivial for $n\neq 6$ and that $\out{S_6}$ is a cyclic group of order 2. 
In \cite{BP} the authors used the notation $\nu_6$ for the automorphism such that its class generates $\out{S_6}$, see \cite[Introduction]{BP} for an explicit definition of this outer automorphism.
\begin{enumerate}
	\item Let $\Sigma_n$ be the set of all surjective homomorphisms from $VB_n$ onto $S_n$, let ${\cal T}_n=\Sigma_n/\!\sim_{c}$ be the set of equivalence classes of $\Sigma_n$ by $\sim_{c}$. 
We choose the following set of representatives $\Lambda_n$ of ${\cal T}_n$:
\begin{itemize}
	\item $\Lambda_3=\{\psi_i\mid 1\leq i \leq 8\}$, from \rethm{surjvb3}, 
	\item $\Lambda_4=\{\delta_i\mid 1\leq i \leq 6\}$, from \rethm{surjvb4},
	\item $\Lambda_6=\{\pi_K, \pi_P, \nu_6\circ \pi_K, \nu_6\circ \pi_P\}$, from \cite[Theorem~2.1]{BP}, and 
	\item $\Lambda_n=\{\pi_K, \pi_P\}$, for $n\geq 5$ and $n\neq 6$ from \cite[Theorem~2.1]{BP}.

\end{itemize}

If $n=6$, from \rerem{aut}, $Ker(\nu_6\circ \pi_K)=Ker(\pi_K)$ and $Ker(\nu_6\circ \pi_P)=Ker(\pi_P)$.
Then, from \relem{notisovbn} and \rethm{mainchar} we get that $VP_n$ is characteristic in $VB_n$ when $n\geq 4$ and that 
$KB_n$ is characteristic in $VB_n$ when $n\geq 3$.

	\item Let $\overline{\Sigma}_n$ be the set of all surjective homomorphisms from $WB_n$ onto $S_n$, let $\overline{{\cal T}}_n=\overline{\Sigma}_n/\!\sim_{c}$ be the set of equivalence classes of $\overline{\Sigma}_n$ by $\sim_{c}$. 
We choose the following set of representatives $\overline{\Lambda}_n$ of $\overline{{\cal T}}_n$:
\begin{itemize}
	\item $\overline{\Lambda}_3=\{\overline{\psi_i}\mid 1\leq i \leq 5\}$, from \rethm{surjwb3}, 
	\item $\overline{\Lambda}_4=\{\overline{\delta_i}\mid 1\leq i \leq 4\}$, from \rethm{surjwb4},
	\item $\overline{\Lambda}_6=\{\overline{\pi_P}, \nu_6\circ \overline{\pi_P}\}$, from \cite[Remark~2.8]{M}, and 
	\item $\overline{\Lambda}_n=\{\overline{\pi_P}\}$, for $n\geq 5$ and $n\neq 6$ from \cite[Remark~2.8]{M}.
\end{itemize}

If  $n=6$, from \rerem{aut},  $Ker(\nu_6\circ \overline{\pi_P})=Ker(\overline{\pi_P})$.
Then, from \relem{notisowbn} and \rethm{mainchar} we get the result of the second item. 
	
	\item The proof of this item is similar to the last one. 
	For the proof we use \rethm{surjuvb3}, \rethm{surjuvb4}, \cite[Theorem~1]{M}, \relem{notisowbn} and \rethm{mainchar}.
\end{enumerate}
\end{proof}

\section{Virtual braid groups and the R$_{\infty}$-property}\label{sec:vbn}

In this section we prove \rethm{mainresult}. 
It will be solved case by case, in three subsections, and using slightly different approaches.
We note that from the presentation of the virtual braid group it follows that
$$
VB_2=WB_2=UVB_2=\setang{\sigma_1, v_1}{v_1^2=1} \cong \Z\ast \Z_2. 
$$
Hence, $VB_2=WB_2=UVB_2$ has the R$_{\infty}$-property, see \cite{GSW}. 
We start by recalling a result that we will use repeatedly to prove \rethm{mainresult}.

\begin{lem}[{\cite[Lemma~6]{MS}}]\label{lem:ses}
Consider an exact sequence of groups
$$
1\to K\to G\to Q\to 1
$$ 
where $K$ is a characteristic subgroup of $G$. Then,
\begin{enumerate}
	\item If $Q$ has the R$_{\infty}$-property, then so does $G$.  
	\item If $Q$ is finite and $K$  has the R$_{\infty}$-property, then $G$ so does.
\end{enumerate}
  
\end{lem}

In the sequel we will also make use of some facts about crystallographic groups. A $n$-dimensional crystallographic group $\Gamma$ is a group which fits is a short exact sequence 
\[ 1 \to \Z^n \to \Gamma \to F \to 1\]
where $F$ is a finite group and $\Z^n$ is maximal abelian in $\Gamma$. Such a short exact sequence induces a representation 
$\phi\colon F \to GL_n(\Z)$ which is called the holonomy representation. In fact, requiring that $\Z^n$ is maximal abelian in $\Gamma$ is equivalent to asking that $\phi$ is a faithful representation. 
For such a $n$-dimensional crystallographic group $\Gamma$ there exists an embedding $\rho\colon \Gamma \to \aff{\R^n}=\R^n \rtimes GL_n(\R)$ with  $\rho(\Gamma)\cap \R^n = \rho(\Z^n) = \Z^n$. After identifying $\Gamma$ with its image $\rho(\Gamma)$ in $\aff{\R^n}$, the second Bieberbach theorem implies that any automorphism $\psi \in \aut{\Gamma}$ can be realised as an affine conjugation, i.e.\ $\exists (d,D) \in \aff{\R^n}$ such that 
$\psi(\gamma) = (d,D) \gamma (d,D)^{-1}$ for all $\gamma \in \Gamma$. We refer the reader to \cite{De} for more details.
\subsection{The case of 3 strings}

We use the presentation of the virtual braid group with 3 strings given in \req{presvb3}.

\begin{lem}\label{lem:vb3cryst}
    The normal closure of the coset of the element $v_1v_2$ in $VB_3/[VP_3,\, VP_3]$ (resp.\ in $VB_3/[KB_3,\, KB_3]$) is a characteristic subgroup of $VB_3/[VP_3,\, VP_3]$ (resp.\ of $VB_3/[KB_3,\, KB_3]$). 
\end{lem}

\begin{proof}
We recall from \cite[Theorem~3.3 and equation~(8)]{CO} that there is a decomposition $VB_3/[VP_3,\, VP_3]\cong VP_3/[VP_3,\, VP_3]\rtimes S_3$ where $VP_3/[VP_3,\, VP_3]$ is the free abelian group of rank 6 generated by $\{ \lambda_{i,j} \mid 1\leq i\neq j\leq 3 \}$, the symmetric group is generated by two transpositions $v_1, v_2$, and such that the action  is given by permutation of indices. 
So, we may write a presentation for $VB_3/[VP_3,\, VP_3]$ with generators $v_1$, $v_2$ and $\lambda_{i,j}$ for $ 1\leq i\neq j\leq 3$ and defining relations given by
\begin{itemize}
	\item $v_1v_2v_1 = v_2v_1v_2$, $v_1^2=1$, $v_2^2=1$,
	\item $[\lambda_{i,j},\, \lambda_{k,l}]=1$ for $1\leq i\neq j\leq 3$ and $1\leq k\neq l\leq 3$,
	\item $v_k\cdot \lambda_{i,j}\cdot v_k = \lambda_{v_k(i),v_k(j)}$, for $1\leq i\neq j\leq 3$ and $k=1,2$.
\end{itemize}
Let $\gamma=v_1v_2\in VB_3/[VP_3,\, VP_3]$. Hence $\gamma$ has order 3 in $VB_3/[VP_3,\, VP_3]$. 
Let $N$ be the normal closure of the element $\gamma$ in $VB_3/[VP_3,\, VP_3]$. 
Since every element in $VB_3/[VP_3,\, VP_3]$ of order 3 is conjugate to $\gamma$ (see \cite[Corollary~3.8]{CO}) we get that $N$ is a characteristic subgroup of $VB_3/[VP_3,\, VP_3]$.

Now, we prove the result for $VB_3/[KB_3,\, KB_3]$. Proposition~17 and Corollary~18 of~\cite{BB} show that $VB_3$ can be seen as a semidirect product
$VB_3=KB_3\rtimes S_3$, where $KB_3$ can be viewed as a group generated by 6 generators $x_{i,j}$ with $1 \leq i\neq j \leq3$ subject to 6 relations $x_{i,k} x_{k,j} x_{i,k} = x_{k,j} x_{k,i} x_{k,j}$ (for $\{i,j,k\}=\{1,2,3\}$) and where $S_3$ acts on the generators by permuting the indices. 
In the proof of Proposition~19 of~\cite{BB} it was shown that in the quotient  $VB_3/[KB_3,\, KB_3]$ these relations lead to an equality of cosets $x_{1,2}=x_{2,3}=x_{3,1}$ and $x_{1,3}=x_{3,2}=x_{2,1}$ (which we abusively also denote by the same symbols).
Hence, from \cite[Proposition~19]{BB}, the group $KB_3/[KB_3,\, KB_3]$ is a free abelian group of rank 2 generated by the cosets of the elements $x_{1,2}$ and $x_{1,3}$ and we obtain a split extension 
\begin{equation*}
    1\to KB_3/[KB_3,\, KB_3]\cong \Z^2 \to VB_3/[KB_3,\, KB_3] \stackrel{\overline{\pi_K}}{\longrightarrow} S_3 \to 1
\end{equation*}
where $\overline{\pi_K}$ is the homomorphism induced from $\pi_K\colon VB_3\to S_3$ (see \resubsec{vbn}). We consider now the following presentation of $S_3$, $S_3=\langle a,\, b \mid a^3=1, b^2=1, (ba)^2=1 \rangle$, where $a=v_1v_2$ and $b=v_1$. 
By using the method described in \cite[Chapter~10]{J} we find a presentation of the group $VB_3/[KB_3,\, KB_3]$ with generators $a$, $b$, $x_{1,2}$, $x_{1,3}$ and defining relations
\begin{itemize}
    \item $a^3=1$; $b^2=1$; $(ba)^2=1$;
    \item $[x_{1,2},\, x_{1,3}]=1$;
    \item $bx_{1,2}b^{-1}=x_{1,3}$; $bx_{1,3}b^{-1}=x_{1,2}$;
    \item $ax_{1,2}a^{-1}=x_{1,2}$; $ax_{1,3}a^{-1}=x_{1,3}$.
\end{itemize}
We consider now the following extension 
\begin{equation*}
    1\to KB_3/[KB_3,\, KB_3] \to \overline{\pi_K}^{-1}(\Z_3) \stackrel{\overline{\pi_K}}{\longrightarrow} \Z_3 \to 1
\end{equation*}
where $\Z_3$ is the group generated by $a=v_1v_2$. Notice that $\overline{\pi_K}^{-1}(\Z_3)$ is isomorphic to $\Z\oplus \Z \oplus \Z_3$ generated by the set $\{ x_{1,2},\, x_{1,3},\, a \}$. 
From the above we obtain the extension
\begin{equation*}
    1\to \Z\oplus \Z \oplus \Z_3 \to VB_3/[KB_3,\, KB_3] \to \Z_2 \to 1
\end{equation*}
where $\Z_2$ is the group generated by $b=v_1$. From this extension we see that the torsion subgroup of $\Z\oplus \Z \oplus \Z_3$ is the unique subgroup of order 3 in $VB_3/[KB_3,\, KB_3]$. So this subgroup, which is generated by $v_1v_2$ is a characteristic subgroup of $VB_3/[KB_3,\, KB_3]$. 
\end{proof}

\relem{vb3cryst} is useful to prove the next result.

\begin{thm}\label{thm:vb3cryst}
The quotient groups $VB_3/[VP_3,\, VP_3]$ and $VB_3/[KB_3,\, KB_3]$ have the R$_{\infty}$-property.
\end{thm}

\begin{proof}
Let $N$ be the normal closure of the coset of $v_1v_2$ in $VB_3/[VP_3,\, VP_3]$. 
We consider the quotient $G=(VB_3/[VP_3,\, VP_3])/N$ that has a presentation given by the one of $VB_3/[VP_3,\, VP_3]$ (see the proof of \relem{vb3cryst}) adding the relation $v_1v_2=1$, which is equivalent to the relation $v_1=v_2$ since $v_1$ and $v_2$ are transpositions. 
From $v_1=v_2$ and the relations $v_k\cdot \lambda_{i,j}\cdot v_k = \lambda_{v_k(i),v_k(j)}$, for $1\leq i\neq j\leq 3$ and $k=1,2$ we conclude that $G\cong \Z^2\rtimes \Z_2$ has a presentation with generators $\lambda_{1,2}$, $\lambda_{2,1}$ and $\sigma_1$ and defining relations
\begin{multicols}{2}
\begin{itemize}
	\item $v_1^2=1$,
	\item $[\lambda_{1,2},\, \lambda_{2,1}]=1$,
	\item $v_1\lambda_{1,2}v_1 = \lambda_{2,1}$,
	\item $v_1\lambda_{2,1}v_1 = \lambda_{1,2}$.
\end{itemize}
\end{multicols}

Let $M$ be the normal closure of the coset of $v_1v_2$ in $VB_3/[KB_3,\, KB_3]$ (which is actually the group of order 3 generated by $v_1v_2$) and let $H=(VB_3/[KB_3,\, KB_3])/M$. From the proof of \relem{vb3cryst} is clear that $H$ is isomorphic to the group $G$ above in this proof.

We note that $G$ and $H$ are isomorphic to the crystallographic group of dimension 2 of Case 5 of the list of all 17 wallpaper groups given in \cite[Section~3]{GW2} (there it was denoted by $G_1^2$). 
Hence, from \cite[Section~3]{GW2}, it follows that $G$ and $H$ have the R$_{\infty}$-property. 
Therefore, from \relem{ses} and \relem{vb3cryst}, we have that $VB_3/[VP_3,\, VP_3]$  and $VB_3/[KB_3,\, KB_3]$ also have the R$_{\infty}$-property.
\end{proof}

With the last result we may prove that $VB_3$, $WB_3$ and $UVB_3$ have the R$_{\infty}$-property.

\begin{cor}\label{cor:vb3}
The virtual braid group $VB_3$, the welded braid group $WB_3$ and the unrestricted virtual braid group $UVB_3$ have the R$_{\infty}$-property.
\end{cor}

\begin{proof}
From \rethm{vb3cryst} we know that the groups $VB_3/[VP_3,\, VP_3]$ and $VB_3/[KB_3,\, KB_3]$ have  the R$_{\infty}$-property. 
From \cite[Theorem~5.1]{CO} the group $VB_3/[VP_3, VP_3]$ is isomorphic to $WB_3/[WP_3, WP_3]$ as well as to $UVB_3/[UVP_3, UVP_3]$. 
Then, by applying \relem{ses} and \rethm{maincharvbn} we get this result. 
\end{proof}

\subsection{The case of 4 strings}

Let $\Z^{12}\rtimes S_4$ be a crystallographic group such that the generators of $\Z^{12}$ are denoted by $\lambda_{i,j}$ for $1\leq i\neq j\leq 4$ and such that the action of $w\in S_4$ on $\lambda_{i,j}$ is given by $w\cdot \lambda_{i,j} = \lambda_{w^{-1}(i), w^{-1}(j)}$ for all $1\leq i\neq j\leq 4$. Let $\phi\colon S_4 \to GL_{12}(\Z)$ be the holonomy representation of $\Z^{12}\rtimes S_4$. 
From the natural homomorphism $GL_{12}(\Z) \hookrightarrow GL_{12}(\Q)$ we shall view the holonomy representation as $\phi\colon S_4\hookrightarrow GL_{12}(\Q)$. 

First, we describe the $S_4$-module structure of $\Q^{12}$ using character theory. We record in Table~\ref{tab:chars4} the character table of $S_4$:
\begin{table}[htb] 
\centering
%\large
\begin{tabular}{| p{3cm} | p{2cm} | p{2cm} | p{2cm} | p{2cm} | p{2cm} | }
\hline
Representation / Conjugacy class representative and size & $(\, )$ Identity element (Size 1) & $(1,\, 2)(3,\, 4)$ (Size 3) & $(1,\, 2)$ (Size 6) & $(1,\, 2,\, 3,\, 4)$ (Size 6) & $(1,\, 2,\, 3)$ (Size 8) \\  \hline
Trivial representation $\chi_1$ & 1 & 1 & 1 & 1 & 1 \\  
\hline
Sign representation $\chi_2$ & 1 & 1 & -1 & -1 & 1 \\  
\hline
Irreducible representation of degree two with kernel of order four $\chi_3$ & 2 & 2 & 0 & 0 & -1 \\  
\hline
Standard representation $\chi_4$ & 3 & -1 & 1 & -1 & 0 \\  
\hline
Product of standard and sign representation $\chi_5$ & 3 & -1 & -1 & 1 & 0 \\  
\hline
\end{tabular}
\caption{The character table of $S_4$}
\label{tab:chars4}
\end{table}

Let $\chi$ be the character of the representation $\phi\colon S_4\hookrightarrow GL_{12}(\Q)$. Recall that, for an element $g\in S_4$, the number $\chi(g)=Tr(\phi(g))$ is equal to the number of generators $\lambda_{i,j}$ that are fixed by $g$.
In Table~\ref{tab:charchi} we show the character $\chi$ evaluated in each of the five conjugacy classes of elements in $S_4$, given by representatives.
\begin{table}[htb] 
\centering
%\large
\begin{tabular}{| p{2.8cm} | p{2.9cm} | p{2.5cm} | p{2.8cm} | p{2.5cm} | }
\hline
$(\, )$ Identity element (Size 1) & $\tau=(1,\, 2)(3,\, 4)$ (Size 3) & $\tau=(1,\, 2)$ (Size 6) & $\tau=(1,\, 2,\, 3,\, 4)$ (Size 6) & $\tau=(1,\, 2,\, 3)$ (Size 8) \\  \hline
$\chi(1)=12$ & $\chi(\tau)=0$ & $\chi(\tau)=2$ & $\chi(\tau)=0$ & $\chi(\tau)=0$  \\  
\hline
\end{tabular}
\caption{The character $\chi\colon S_4 \to GL_{12}(\Q)$}
\label{tab:charchi}
\end{table}

Now, we compute the components of the character $\chi$:
$$
\begin{array}{rclcl}
(\chi \mid \chi_1) & = & \frac{1}{24}(12\cdot 1 + 6\cdot 2\cdot 1) & = & 1\\
(\chi \mid \chi_2) & = & \frac{1}{24}(12\cdot 1 + 6\cdot 2\cdot (-1)) & = & 0\\
(\chi \mid \chi_3) & = & \frac{1}{24}(12\cdot 2 + 6\cdot 2\cdot 0) & = & 1\\
(\chi \mid \chi_4) & = & \frac{1}{24}(12\cdot 3 + 6\cdot 2\cdot 1) & = & 2\\
(\chi \mid \chi_5) & = & \frac{1}{24}(12\cdot 3 + 6\cdot 2\cdot (-1)) & = & 1\\
\end{array}
$$
Hence, the character $\chi$ has the decomposition
\begin{equation}\label{eq:decompchi}
\chi = \chi_1 + \chi_3 + 2\chi_4 + \chi_5.
\end{equation}

Let $V\subseteq \Q^{12}$ be the submodule of $\Q^{12}$ corresponding to $\chi_1 + \chi_3 + 2\chi_4$. 
Then, $V'=V\cap \Z^{12}$ is a submodule of $\Z^{12}$ and so a normal subgroup of $\Z^{12}\rtimes S_4$ such that $\Z^{12}/V'$ is torsion-free.

It follows that, as groups, we can write $\Z^{12} = V' \oplus W'$ where both $V'\cong \Z^9$ and $W'\cong \Z^3$ are free abelian.

\begin{lem}\label{lem:vprime}
The group $V'$ is a characteristic subgroup of $\Z^{12}\rtimes S_4$. 
\end{lem}

\begin{proof}
Let $\{ e_1,e_2,\ldots,e_{12} \}$ be a generating set of $\Z^{12}$ such that $\{ e_1,e_2,\ldots,e_9 \}$ generate $V'$. 
Let $w\in S_4$. 
With respect to this generating set we can write $\phi(w)$ as a $12\times 12$ matrix and since $V'$ is a submodule, we have that 
$$
\phi(w) = 
\begin{pmatrix} 
\sigma_1(w) & \alpha(w) \\
0 & \sigma_2(w) 
\end{pmatrix}
$$
with $\sigma_1\colon S_4\to GL_9(\Z)$ corresponding to $\chi_1 + \chi_3 + 2\chi_4$ and $\sigma_2\colon S_4\to GL_3(\Z)$ corresponding to $\chi_5$. 

We can embed $\Z^{12}\rtimes S_4$ into $\aff{\R^{12}}=\R^{12}\rtimes GL_{12}(\R)$ by mapping $(z,w)$ to $(z,\phi(w))$. 
Let $\psi\in \aut{\Z^{12}\rtimes S_4}$. Note that $\Z^{12}$ is characteristic. 
So, $\psi$ induces an automorphism $\overline{\psi}$ on $S_4$. We know that $\aut{S_4}=\inn{S_4}$.
Hence, there is an inner automorphism $\mu \in \inn{\Z^{12}\rtimes S_4}$ such that $\psi\circ \mu$ induces the identity on $S_4$. 

As $V' \triangleleft \Z^{12}\rtimes S_4$ we know that $\mu(V')=V'$. So, we may assume from now onwards that $\psi$ induces the identity on $S_4$. 
As mentioned before $\psi$ is realized by an affine conjugation. So, there exists an element $(d,D)\in \aff{\R^{12}}$ so that $\psi(z,\phi(w))=(d,D)(z,\phi(w))(d,D)^{-1}$. 
As $\psi$ induces the identity on $S_4$, we must have that $D\phi(w)D^{-1}=\phi(w)$, for all $w\in S_4$. 

Write $D$ as
$
\begin{pmatrix} 
D_1 & D_2 \\
D_3 & D_4 
\end{pmatrix}.
$
Therefore,
$$
\begin{pmatrix} 
D_1 & D_2 \\
D_3 & D_4 
\end{pmatrix}
\begin{pmatrix} 
\sigma_1(w) & \alpha(w) \\
0 & \sigma_2(w) 
\end{pmatrix}
=
\begin{pmatrix} 
\sigma_1(w) & \alpha(w) \\
0 & \sigma_2(w) 
\end{pmatrix}
\begin{pmatrix} 
D_1 & D_2 \\
D_3 & D_4 
\end{pmatrix}
$$
and so
\begin{equation}\label{eq:d3sigma}
D_3\sigma_1(w) = \sigma_2(w) D_3, \textrm{ for all $w\in S_4$.}
\end{equation}

Notice that $D_3$ is a $3\times 9$ matrix and can be viewed as a map $D_3\colon \Q^9\to \Q^3$, with $\Q^9$ an $S_4$-module via $\sigma_1$ and $\Q^3$ an $S_4$-module via $\sigma_2$. 
Equation~\reqref{d3sigma} shows that $D_3$ is an $S_4$-module map from $\Q^9$ to $\Q^3$, where $\Q^3$ is an irreducible module and $\Q^9$ does not contain a submodule isomorphic to $\Q^3$ and so $D_3 = 0$.
Hence  $
D=\begin{pmatrix} 
D_1 & D_2 \\
0 & D_4 
\end{pmatrix}.
$

It now follows that for $z\in V'$ we have that 
$$
\begin{array}{rcl}
 \psi(z) & = & (d,D)(z,1)(d,D)^{-1}\\
 & = & (d+Dz, D)(-D^{-1}d,D^{-1})\\
 & = & (d+Dz -d, 1)\\
 & = & (Dz,1).
\end{array}
$$
But, since $
D=\begin{pmatrix} 
D_1 & D_2 \\
0 & D_4 
\end{pmatrix}
$ 
it follows that $Dz\in V'$. 
\end{proof}

\begin{thm}\label{thm:crystgroup}
The group $\Z^{12}\rtimes S_4$ has the R$_{\infty}$-property.
\end{thm}

\begin{proof}
The quotient of $\Z^{12}\rtimes S_4$ by the characteristic subgroup $V'$ of \relem{vprime} satisfies
$$
\Z^{12}/V'\rtimes S_4 \cong \Z^3\rtimes S_4
$$
where the action is faithful (it corresponds to $\chi_5$). Then it is a 3-dimensional crystallographic group. From \cite[Theorem~4.2]{DP} we know that this group has the 
R$_{\infty}$-property. 
Hence, from \relem{ses} the result follows. 
\end{proof}

\begin{cor}\label{cor:excep}
The virtual braid group $VB_4$, the welded braid group $WB_4$ and the unrestricted virtual braid group $UVB_4$ have the R$_{\infty}$-property.
\end{cor}

\begin{proof}
We note that the group $VB_4/[VP_4, VP_4]$ is isomorphic to the group $\Z^{12}\rtimes S_4$ of \rethm{crystgroup} (see \cite[Theorem~3.3]{CO}), so by \rethm{crystgroup} it has the R$_{\infty}$-property. 
From \cite[Theorem~5.1]{CO}  $VB_4/[VP_4, VP_4]$ is isomorphic to $WB_4/[WP_4, WP_4]$ as well as to $UVB_4/[UVP_4, UVP_4]$. 
Then, from \relem{ses} and \rethm{maincharvbn} we conclude the result for this corollary. 
\end{proof}

\subsection{General cases} 

In the next proposition we show that, for $n\geq 2$, the R$_{\infty}$-property holds for the unrestricted virtual braid group $UVB_n$ and its pure subgroup $UVP_n$. Then we use it to prove the result for the virtual braid group $VB_n$, with $n\geq 5$.

\begin{prop}\label{prop:uvbn}
Let $n\geq 2$. 
The unrestricted virtual pure braid group $UVP_n$ and the unrestricted virtual braid group $UVB_n$ have the R$_{\infty}$-property.
\end{prop}

\begin{proof}
The case $n=2$ for $UVB_n$ was mentioned in the first paragraph of this section. 
From \cite[Remark~2.10]{M} it follows that, for $n\geq 2$, $UVP_n$ is isomorphic to the direct product of $n(n-1)/2$ copies of the free group of rank 2. Hence, from \cite[Example~5.1.8]{DS} we conclude that $UVP_n$ has the R$_{\infty}$-property.

Now, let $n\geq 3$. 
From \rethm{maincharvbn}, the group $UVP_n$ is a characteristic subgroup of $UVB_n$ (see also \cite[Proposition~2.15]{M} for $n\geq 5$). Then, from \relem{ses} applied to the short exact sequence $1\to UVP_n\to UVB_n\to S_n\to 1$ we obtain the result for $UVB_n$.
\end{proof}
 
\begin{rem}
We note that, for $n=3$ and $4$, we also proved the R$_{\infty}$-property for $UVB_n$ in Corollaries~\ref{cor:vb3} and~\ref{cor:excep}, respectively, but using different techniques.
\end{rem}

All possible homomorphisms from $VB_n$ to $VB_m$ were determined in \cite[Theorem~2.3]{BP}, for $n\geq 5$, $m\geq 2$ and $n\geq m$.  
In particular, for $n\geq 5$, $\out{VB_n}$, the outer automorphism group of $VB_n$,  is isomorphic to $\Z_2\times \Z_2$ and is generated by the classes of $\zeta_1$ and $\zeta_2$ where
\begin{itemize}
	\item $\zeta_1\colon VB_n\to VB_n$ is defined by $\zeta_1(\sigma_i)=v_i\sigma_iv_i$ and $\zeta_1(v_i)=v_i$;
	
	\item $\zeta_2\colon VB_n\to VB_n$ is defined by $\zeta_2(\sigma_i)=\sigma_i^{-1}$ and $\zeta_2(v_i)=v_i$.
\end{itemize}
for $i=1,\ldots,n-1$, see \cite[Corollary~2.5]{BP}.

\begin{lem}\label{lem:key}
Let $n\geq 5$. 
The normal closure $K$ of the set 
$$
\{v_i\sigma_{i+1}\sigma_iv_{i+1}\sigma_i^{-1} \sigma_{i+1}^{-1} ;\,  v_{i+1}\sigma_i \sigma_{i+1}v_i\sigma_{i+1}^{-1}\sigma_i^{-1}  \mid  i=1,\dots,n-2\}
$$ 
is a characteristic subgroup of $VB_n$.
\end{lem}

\begin{proof}

We shall use the presentation of $VB_n$ given in Definition~\ref{apvbn}. 
Recall that, for all $i=1,\ldots,n-1$, $v_i=v_i^{-1}$ in $VB_n$. 
In the following computations we use the mixed relation (MR2) $v_iv_{i+1}\sigma_{i}=\sigma_{i+1}v_{i}v_{i+1}$ of $VB_n$, for $i=1,2,\dots,n-2$,  which is equivalent to 
$\sigma_{i}v_{i+1}v_i = v_{i+1}v_{i}\sigma_{i+1}$ or to 
$v_iv_{i+1}\sigma_{i}^{-1}=\sigma_{i+1}^{-1}v_{i}v_{i+1}$ or to 
$\sigma_{i}^{-1}v_{i+1}v_i = v_{i+1}v_{i}\sigma_{i+1}^{-1}$.

$$
\begin{array}{rcl}

\zeta_1(v_i\sigma_{i+1}\sigma_iv_{i+1}\sigma_i^{-1}\sigma_{i+1}^{-1}) & = & 
v_iv_{i+1}\sigma_{i+1}v_{i+1} v_{i}\sigma_i\underline{v_{i} v_{i+1} v_{i}}\sigma_i^{-1}v_{i} v_{i+1}\sigma_{i+1}^{-1}v_{i+1}\\
 & = & v_iv_{i+1}\sigma_{i+1}v_{i+1} v_{i}\underline{\sigma_iv_{i+1} v_{i}} \underline{v_{i+1}\sigma_i^{-1}}v_{i} v_{i+1}\sigma_{i+1}^{-1}v_{i+1}\\
 & = & v_iv_{i+1}\sigma_{i+1}\underline{v_{i+1} v_{i}v_{i+1}v_{i}}\sigma_{i+1}   v_i\sigma_{i+1}^{-1}\underline{v_{i}v_{i+1}v_{i} v_{i+1}}\sigma_{i+1}^{-1}v_{i+1}\\
 & = & v_iv_{i+1}\underline{ \sigma_{i+1}v_{i}v_{i+1}}\sigma_{i+1} v_i\sigma_{i+1}^{-1}\underline{v_{i+1}v_{i}\sigma_{i+1}^{-1}}v_{i+1}\\
 & = & \underline{v_iv_{i+1}v_i}v_{i+1}\sigma_{i}\sigma_{i+1} v_i\sigma_{i+1}^{-1}\sigma_{i}^{-1}v_{i+1}v_iv_{i+1}\\
 & = & v_{i+1}v_iv_{i+1}\cdot v_{i+1}\sigma_{i}\sigma_{i+1} v_i\sigma_{i+1}^{-1}\sigma_{i}^{-1} \cdot v_{i+1}v_iv_{i+1} 
\end{array}
$$

$$
\begin{array}{rcl}

\zeta_1(v_{i+1}\sigma_i\sigma_{i+1}v_i\sigma_{i+1}^{-1}\sigma_i^{-1}) & = & 
v_{i+1}v_{i}\sigma_iv_{i}v_{i+1}\sigma_{i+1}\underline{v_{i+1}v_i v_{i+1}}\sigma_{i+1}^{-1}v_{i+1}v_{i}\sigma_i^{-1}v_{i}\\
 & = & v_{i+1}v_{i}\sigma_iv_{i}v_{i+1}\underline{\sigma_{i+1}v_iv_{i+1}}   \underline{v_i\sigma_{i+1}^{-1}}v_{i+1}v_{i}\sigma_i^{-1}v_{i}\\
 & = & v_{i+1}v_{i}\sigma_i\underline{v_{i}v_{i+1}v_iv_{i+1}}\sigma_{i}    v_{i+1}\sigma_{i}^{-1}\underline{v_{i+1}v_iv_{i+1}v_{i}}\sigma_i^{-1}v_{i}\\
 & = & v_{i+1}v_{i}\underline{\sigma_iv_{i+1}v_i}\sigma_{i}    v_{i+1}\sigma_{i}^{-1}\underline{v_iv_{i+1}\sigma_i^{-1}}v_{i}\\
 & = & \underline{v_{i+1}v_{i}v_{i+1}}v_{i}\sigma_{i+1}\sigma_{i}v_{i+1}\sigma_{i}^{-1}\sigma_{i+1}^{-1}v_{i}v_{i+1}v_{i}\\
 & = & v_{i}v_{i+1}v_{i} \cdot v_{i}\sigma_{i+1}\sigma_{i}v_{i+1}\sigma_{i}^{-1}\sigma_{i+1}^{-1}\cdot v_{i}v_{i+1}v_{i}\\
\end{array}
$$

$$
\begin{array}{rcl}

\zeta_2(v_i\sigma_{i+1}\sigma_iv_{i+1}\sigma_i^{-1}\sigma_{i+1}^{-1}) & = & v_i\sigma_{i+1}^{-1}\sigma_i^{-1}v_{i+1}\sigma_i\sigma_{i+1}\\
 & = & (v_{i+1}\sigma_i\sigma_{i+1})^{-1}v_{i+1}\sigma_i\sigma_{i+1}v_i\sigma_{i+1}^{-1}\sigma_i^{-1}(v_{i+1}\sigma_i\sigma_{i+1})
\end{array}
$$

$$
\begin{array}{rcl}

\zeta_2(v_{i+1}\sigma_i\sigma_{i+1}v_i\sigma_{i+1}^{-1}\sigma_i^{-1}) & = & v_{i+1}\sigma_i^{-1}\sigma_{i+1}^{-1}v_i\sigma_{i+1}\sigma_i\\
 & = & (v_i\sigma_{i+1}\sigma_i)^{-1}v_i\sigma_{i+1}\sigma_iv_{i+1}\sigma_i^{-1}\sigma_{i+1}^{-1}(v_i\sigma_{i+1}\sigma_i)
\end{array}
$$

\end{proof}

We recall that the unrestricted virtual braid group $UVB_n$ is the quotient group $VB_n/K$ of the virtual braid group, see \redef{uvbn}.

\begin{thm}\label{main}
Let $n\geq 5$. The virtual braid group $VB_n$ has the R$_{\infty}$-property.
\end{thm}

\begin{proof}
From \relem{key} we know that $K$ is characteristic in $VB_n$. The quotient $VB_n/K$ is the unrestricted virtual braid group $UVB_n$ that, from \reprop{uvbn}, has the R$_{\infty}$-property. The desired result then follows by applying \relem{ses}.
\end{proof}

\begin{rem}

\begin{enumerate}
	\item To the best of our knowledge, for $n\geq 5$, it is not known if the welded braid group $WB_n$ has the R$_{\infty}$-property.
	
	\item From the presentation of $VP_n$ (see \cite[Theorem~1]{B}) we get $VP_2\cong \Z \ast \Z$, which we know it has the R$_{\infty}$-property. Since  $VP_2=WP_2=UVP_2$, these groups  have the R$_{\infty}$-property. 
	For $n\geq 3$, as far as we know it is unknown if the R$_{\infty}$-property holds or not for virtual pure braid groups 
	and for the welded pure braid groups. 
\end{enumerate}

\end{rem}

As explained at the end of the introduction, \rethm{mainresult} is now proved by collecting all the results of this section.

\section{Appendix}

We note that the technique used in this work to prove that some subgroups of virtual braid groups are characteristic may be applied to other braid-like groups. 
We exemplify it in this section by showing that some remarkable subgroups of virtual twin groups are characteristic. For more details about these groups see \cite{NNS1} and \cite{NNS2} and the references therein.

The {virtual twin group} $VT_n$, $n\ge 2$, admits a presentation with generators $\sigma_i,\rho_i$ for $i=1,\dots,n-1$ and defining relations:
\begin{itemize}
\item $\sigma_i^2=1$ for $i=1,2,\dots,n-1$.
\item $\sigma_i\sigma_j=\sigma_j\sigma_i$ for $|i-j|\ge 2$.
\item $\rho_i^2=1$ for $i=1,\dots,n-1$.
\item $\rho_i\rho_j=\rho_j\rho_i$ for $|i-j|\ge 2$.
\item $\rho_i\rho_{i+1}\rho_{i}=\rho_{i+1}\rho_{i}\rho_{i+1}$, for $i=1,2,\dots,n-2$.
\item $\rho_i\sigma_j=\sigma_j\rho_i$, for $|i-j|\ge 2$.
\item $\rho_i\rho_{i+1}\sigma_{i}=\sigma_{i+1}\rho_{i}\rho_{i+1}$, for $i=1,\dots,n-2$.
\end{itemize}

Let $n\geq 2$. For $1\leq i\leq n-1$ let $\tau_i=(i,i+1)$ as before. 
Let $\pi_P\colon VT_n\longrightarrow S_n$ be the homomorphism defined by ${\pi_P}(\sigma_i)={\pi_P}(\rho_i)=\tau_i$ for $i=1,\dots,n-1$. 
The \textit{pure virtual twin group} $PVT_n$ is defined to be the kernel of $\pi_P$. 
Let $\pi_K\colon VT_n\longrightarrow S_n$ be the homomorphism defined by ${\pi_K}(\sigma_i)=1$ and ${\pi_K}(\rho_i)=\tau_i$ for $i=1,\dots,n-1$. 
The kernel of $\pi_K$ will be denoted by $KT_n$.

\begin{thm}\label{thm:charvtn}
Let $n\geq 2$. 
\begin{enumerate}
	\item The groups $PVT_n$ and $KT_n$ are not isomorphic.
	
	\item The virtual pure twin group $PVT_n$ is a characteristic subgroup of the virtual twin group $VT_n$ if and only if $n\neq 3$ and the group $KT_n$ is a  characteristic subgroup of $VT_n$ if and only if $n\geq 3$.
 
\end{enumerate}

\end{thm}

\begin{proof}
Let $n\geq 2$. 

\begin{enumerate}
	\item From \cite{NNS1} the pure virtual twin group $PVT_n$ is a right-angled Artin group (hence it is torsion free) and from \cite{NNS2} the group $KT_n$ is a right-angled Coxeter group (so it has finite order elements), hence they are not isomorphic. 

\item The proof of this item follows the same lines as the proof of \rethm{maincharvbn}, so we just indicate some steps of the proof.

\textit{Claim~1}: There are, up to conjugation, 
\begin{enumerate}
	\item 3 surjective homomorphisms from $VT_2$ to $S_2$;
	\item 5 surjective homomorphisms from $VT_3$ to $S_3$;
	\item 6 surjective homomorphisms from $VT_4$ to $S_4$;
	\item 4 surjective homomorphisms from $VT_6$ to $S_6$; and
	\item 2 surjective homomorphisms from $VT_n$ to $S_n$, for $n\geq 5$ and $n\neq 6$.
\end{enumerate}
The proof of \textit{Claim~1} for $n\geq 5$ may be found in \cite{NNS2}. For the cases $n=2,3,4$ the verification is done as in \resec{char} for the virtual braid group. 

\textit{Claim~2}: It is clear that the image of $KT_2$ by the automorphism $\psi\colon VT_2\to VT_2$ defined by $\psi(\sigma_1)=\rho_1$ and $\psi(\rho_1)=\sigma_1$ is not $KT_2$. 
Also, it is not difficult to verify that the image of $PVT_3$ by the automorphism $\phi\colon VT_3\to VT_3$ defined by $\phi(\sigma_1)=\sigma_2$, $\phi(\sigma_2)=\rho_1\rho_2\sigma_2\rho_2\rho_1$, $\phi(\rho_1)=\rho_1$ and $\phi(\rho_2)=\rho_2$ is not $PVT_3$. 
Therefore, the group $KT_2$ (resp.\ $PVT_3$) is not a characteristic subgroup of $VT_2$ (resp.\ $VT_3$).  

\textit{Claim~3}: 
The groups $PVT_2$, $PVT_4$, $KT_3$ and $KT_4$ are not isomorphic to the kernels of the other homomorphisms (for the same number of strings) from Claim 1. 
The verification of this claim can be done along the same lines as we did for $VB_n$ in \resec{char}.

Then, applying \rethm{mainchar}, we get that for $n\neq 3$ (resp. $n\geq 3$) the groups $PVT_n$ (resp. $KT_n$) are characteristic subgroups of $VT_n$. 
\end{enumerate}
\end{proof}

An application of \rethm{charvtn} is the following result.

\begin{cor}
Let $n\geq 2$. 
The virtual twin group $VT_n$ has the R$_{\infty}$-property. 
\end{cor}

\begin{proof}

Since $VT_2$ is isomorphic to $\Z_2 \ast \Z_2$ then from \cite[Proposition 2.8]{GW2} (see also \cite[Lemma~2]{GSW}) it has the R$_{\infty}$-property. 
From the presentation of $KT_n$ given in \cite[Theorem~3.3]{NNS2} we get the isomorphism $KT_3\cong \Z_2\ast \Z_2\ast \Z_2\ast \Z_2\ast \Z_2\ast \Z_2$. From \cite[Lemma~2]{GSW} the group $KT_3$ has the R$_{\infty}$-property and since it is a characteristic subgroup of $VT_3$ (\rethm{charvtn}) then from \relem{ses} the latter group also  has the R$_{\infty}$-property. 
For $n\geq 4$ the pure virtual twin group $PVT_n$ has the R$_{\infty}$-property, see \cite{NNS1}. 
Then, from \relem{ses} and \rethm{charvtn}, the virtual twin group $VT_n$ has the R$_{\infty}$-property.
\end{proof}

\end{document}